\documentclass{amsart}
\usepackage{amsmath,amscd,xypic,amssymb,combelow,color,enumitem,graphicx,float}
\usepackage{comment}
\usepackage[normalem]{ulem}
\usepackage[dvipsnames]{xcolor}
\makeatletter
\def\squiggly{\bgroup \markoverwith{\textcolor{black}{\lower3.5\p@\hbox{\sixly \char58}}}\ULon}

\usepackage{xr-hyper}
\usepackage[
pdftex,
bookmarks=false,
colorlinks=true,
debug=true,
pdfnewwindow=true]{hyperref}

\makeatletter
  \@addtoreset{equation}{section}
\makeatother

\newtheorem{theorem}[subsection]{Theorem}

\newtheorem{proposition}[subsection]{Proposition}
\newtheorem{lemma}[subsection]{Lemma}
\newtheorem{corollary}[subsection]{Corollary}

\newtheorem{definition}[subsection]{Definition}

\theoremstyle{remark}
\newtheorem{claim}[subsection]{Claim}

\newtheorem{remark}[subsection]{Remark}

\def\fa{{\mathfrak{a}}}
\def\fb{{\mathfrak{b}}}
\def\fg{{\mathfrak{g}}}
\def\fh{{\mathfrak{h}}}

\def\fn{{\mathfrak{n}}}

\def\hg{{\widehat{\fg}}}

\def\BN{{\mathbb{N}}}

\def\BR{{\mathbb{R}}}
\def\BQ{{\mathbb{Q}}}
\def\BZ{{\mathbb{Z}}}

\def\CF{{\mathcal{F}}}

\def\CI{{\mathcal{I}}}

\def\CT{{\mathcal{T}}}

\def\ph{\varphi}

\def\uu{U_q(\fg)}
\def\uup{U_q(\fn^+)}

\def\uuo{U_q(\fh)}
\def\uum{U_q(\fn^-)}

\def\UU{U_q(L\fg)}
\def\UUp{U_q(L\fn^+)}
\def\UUpm{U_q(L\fn^\pm)}
\def\UUo{U_q(L\fh)}
\def\UUm{U_q(L\fn^-)}

\def\VV{U_q(\widehat{\fg})}
\def\VVp{U_q(\widehat{\fn}^+)}
\def\VVo{U_q(\widehat{\fh})}
\def\VVm{U_q(\widehat{\fn}^-)}
\def\VVpm{U_q(\widehat{\fn}^\pm)}
\def\VVg{U_q(\widehat{\fb}^+)}
\def\VVl{U_q(\widehat{\fb}^-)}

\def\sfe{{\mathsf{e}}}

\def\bs{\alpha}

\def\bare{e}

\def\slaws{\text{standard Lyndon loop words}}

\newcommand\iso{\,\vphantom{j^{X^2}}\smash{\overset{\sim}{\vphantom{\rule{0pt}{0.20em}}\smash{\longrightarrow}}}\,}

\def\hdeg{\text{hdeg}}
\def\vdeg{\text{vdeg }}

\def\wI{\widehat{I}}
\def\wQ{\widehat{Q}}
\def\wW{\widehat{W}}
\def\weW{\widehat{W}^{\mathrm{ext}}}

\def\wmu{\widehat{\mu}}
\def\wDelta{\widehat{\Delta}}

\def\talpha{\tilde{\alpha}}
\def\tbeta{\tilde{\beta}}

\def\obeta{\overline{\beta}}
\def\ogamma{\overline{\gamma}}
\def\oalpha{\overline{\alpha}}

\newcommand{\hooklongrightarrow}{\lhook\joinrel\longrightarrow}

\def\hgt{\text{ht}}
\def\wt{\widetilde}
\def\hdeg{\text{hdeg }}
\def\vdeg{\text{vdeg }}
\def\ff{\mathsf{f}}

\begin{document}

\title[Standard Lyndon loop words]
      {Standard Lyndon loop words: weighted orders}
	
\author[S.~Khomych, N.~Korniichuk, K.~Molokanov, A.~Tsymbaliuk]
       {Severyn Khomych, Nazar Korniichuk, Kostiantyn Molokanov,\\ and Alexander Tsymbaliuk}

\address{S.K.: University of Vienna, Department of Mathematics, Vienna, Austria}
\email{severyn.khomych@gmail.com}

\address{N.K.: MIT, Department of Mathematics, Cambridge, MA, USA}
\email{n.korniychuk.a@gmail.com}

\address{K.M.: Technical University of Berlin, Department of Mathematics, Berlin, Germany}
\email{kostyamolokanov@gmail.com}

\address{A.T.: Purdue University, Department of Mathematics, West Lafayette, IN, USA}
\email{sashikts@gmail.com}

\begin{abstract}
We generalize the study of standard Lyndon loop words from~\cite{NT} to a more general class of
orders on the underlying alphabet, as suggested in \cite[Remark 3.15]{NT}. The main new ingredient is
the exponent-tightness of these words, which also allows to generalize the construction of PBW bases
of the untwisted quantum loop algebra $\UU$ via the combinatorics of loop words.
\end{abstract}

\maketitle


\section{Introduction}


\subsection{Summary}\label{sub:summary}
\

An interesting basis of the free Lie algebra generated by a finite family $\{e_i\}_{i\in I}$ was constructed
in the 1950s using the combinatorial notion of \emph{Lyndon} words. A few decades later, this was generalized
to any finitely generated Lie algebra $\fa$ in~\cite{LR}. Explicitly, if $\fa$ is generated by $\{e_i\}_{i\in I}$,
then any order on the finite alphabet $I$ gives rise to the combinatorial basis $e_{\ell}$ as $\ell$ ranges
through all \emph{standard Lyndon} words.

The key application of~\cite{LR} was to simple finite-dimensional $\fg$, or more precisely,
to its maximal nilpotent subalgebra $\fn^+$. According to the root space decomposition:
\begin{equation}\label{eqn:root vectors intro}
  \fn^+ = \bigoplus_{\alpha \in \Delta^+} \BQ \cdot e_\alpha \,, \qquad
  \Delta^+=\Big\{\mathrm{positive\ roots}\Big\} ,
\end{equation}
with elements $e_\alpha$ called \emph{root vectors}. By the PBW theorem, we thus have
\begin{equation}\label{eqn:pbw intro 1}
  U(\fn^+) \ =
  \bigoplus^{k\in \BN}_{\gamma_1 \geq \dots \geq \gamma_k \in \Delta^+}
    \BQ \cdot e_{\gamma_1} \dots e_{\gamma_k}
\end{equation}
for any total order on $\Delta^+$, with $\BN=\BZ_{\geq 0}$. Furthermore, a triangular decomposition
\begin{equation}\label{eqn:decomp intro 1}
  \fg = \fn^+ \oplus \fh \oplus \fn^-
\end{equation}
induces the corresponding triangular decomposition of the universal enveloping:
\begin{equation}\label{eqn:decomp intro 1.2}
  U(\fg) = U(\fn^+) \otimes U(\fh) \otimes U(\fn^-) \,.
\end{equation}
Moreover, the root vectors satisfy ($R^*$ shall denote nonzero elements of a ring $R$)
\begin{equation}\label{eqn:comm intro 1}
  [e_{\alpha} , e_{\beta}] = e_{\alpha} e_{\beta} - e_{\beta} e_{\alpha} \in \BQ^* \cdot e_{\alpha + \beta}
\end{equation}
whenever $\alpha,\beta\in \Delta^+$ satisfy $\alpha+\beta\in \Delta^+$. Thus, formula~\eqref{eqn:comm intro 1}
provides an algorithm for constructing all the root vectors \eqref{eqn:root vectors intro} inductively starting
from $e_i = e_{\alpha_i}$, where $\{\alpha_i\}_{i\in I}\subset \Delta^+$ are the simple roots of $\fg$.
Therefore, all the root vectors $\{e_\alpha\}_{\alpha\in \Delta^+}$, and hence the PBW basis \eqref{eqn:pbw intro 1},
can be read off from the combinatorics of $\Delta^+$.

\medskip
\noindent
The above discussion can be naturally adapted to the quantizations. Let $\uu$ be the Drinfeld-Jimbo quantum group
of $\fg$, a $q$-deformation of the universal enveloping algebra $U(\fg)$. For one thing, it admits a triangular
decomposition similar to~\eqref{eqn:decomp intro 1.2}:
\begin{equation}\label{eqn:decomp intro 2}
  \uu = \uup \otimes U_q(\fh) \otimes \uum \,.
\end{equation}
Here, $\uup$ is the \emph{positive} subalgebra of $\uu$, explicitly generated by $\{\wt{e}_i\}_{i\in I}$
subject to $q$-Serre relations. There exists a PBW basis analogous to~\eqref{eqn:pbw intro 1}:
\begin{equation*}
  \uup \ =
  \bigoplus^{k\in \BN}_{\gamma_1 \geq \dots \geq \gamma_k \in \Delta^+}
    \BQ(q) \cdot \wt{e}_{\gamma_1} \dots \wt{e}_{\gamma_k} \,.
\end{equation*}
The $q$-deformed root vectors $\wt{e}_\alpha \in \uup$ are defined via Lusztig's braid group action,
which requires one to choose a reduced decomposition of the longest element in the Weyl group of $\fg$.
It is well-known (\cite{P}) that this choice precisely ensures that the order $\geq$ on $\Delta^+$ is
\emph{convex}, in the sense of Definition \ref{def:convex}. Moreover, as follows from the
Levendorsky-Soibelman property~\cite{LS}, the $q$-deformed root vectors satisfy the following $q$-analogue
of the relation~\eqref{eqn:comm intro 1}:
\begin{equation}\label{eqn:comm intro 2}
  [\wt{e}_{\alpha} , \wt{e}_{\beta}]_q =
  \wt{e}_{\alpha} \wt{e}_{\beta} - q^{(\alpha, \beta)} \wt{e}_{\beta} \wt{e}_{\alpha}
  \in \BQ(q)^* \cdot \wt{e}_{\alpha+\beta}
\end{equation}
whenever $\alpha, \beta, \alpha+\beta \in \Delta^+$ satisfy $\alpha<\alpha+\beta<\beta$
as well as the \emph{minimality} property
\begin{equation*}
  \not \exists \ \alpha',\beta' \in \Delta^+ \quad \text{s.t.} \quad
  \alpha < \alpha' < \beta' < \beta \quad \text{and} \quad \alpha+\beta = \alpha'+\beta' \,,
\end{equation*}
and $(\cdot, \cdot)$ denotes the scalar product corresponding to the root system of type $\fg$.
Thus, similarly to the Lie algebra case, we conclude that the $q$-deformed root vectors can be defined
(up to scalar multiples) as iterated $q$-commutators of $\wt{e}_i = \wt{e}_{\alpha_i}$ ($i\in I$),
using the combinatorics of $\Delta^+$ and the chosen convex order on it.

\medskip
\noindent
Following~\cite{G,R1,S}, let us recall that $\uup$ can be also defined as a subalgebra of
the $q$-shuffle algebra:
\begin{equation*}
  \uup \stackrel{\Phi}\hooklongrightarrow \CF \ =
  \bigoplus^{k\in \BN}_{i_1,\dots,i_k \in I} \BQ(q) \cdot [i_1 \dots i_k] \,,
\end{equation*}
where $\CF$ has a basis $I^*$, consisting of finite length words in $I$, and is endowed with
the \emph{quantum shuffle} product. As mentioned above, there is a natural bijection
\begin{equation}\label{eqn:1-to-1 intro}
  \ell \colon \Delta^+ \iso \Big\{\text{standard Lyndon words}\Big\} ,
\end{equation}
established in~\cite{LR}. This induces the lexicographic order on $\Delta^+$ via
\begin{equation*}
  \alpha < \beta \quad \Longleftrightarrow \quad \ell(\alpha) < \ell(\beta) \text{ lexicographically} \,.
\end{equation*}
As shown in~\cite{L,R2} this total order is convex, and hence can be applied to obtain quantum root vectors
$\wt{e}_\alpha \in \uup$ for any positive root $\alpha$, as in~\eqref{eqn:comm intro 2}. Moreover, \cite{L}
shows that the quantum root vector $\wt{e}_\alpha$ is uniquely characterized (up to a scalar multiple) by
the property that $\Phi(e_\alpha)$ is an element of $\text{Im }\Phi$ whose leading order term $[i_1 \dots i_k]$
(in the lexicographic order) is precisely~$\ell(\alpha)$.

\medskip
\noindent
It is natural to ask if the above results can be generalized from simple $\fg$ to affine Lie algebras
$\widehat{\fg}$. The main complication arises from the fact that not all root subspaces of $\widehat{\fg}$
are one-dimensional. In~\cite{AT}, an analogue of~\eqref{eqn:1-to-1 intro} was established and all standard
Lyndon words were explicitly computed for $\widehat{\fg}$ with $\fg$ of $A$-type. On the other hand,
considering a different (new Drinfeld) ``polarization'' of quantum loop algebras
  $$ \UU = \UUp \otimes \UUo \otimes \UUm \,, $$
the above complication disappears as $\UUp$ is a $q$-deformation of the universal enveloping algebra
of $\fn^+[t,t^{-1}]$ all of which root subspaces are one-dimensional. In particular, many of the above
results were adapted to the loop setup in~\cite{NT}.

\medskip
\noindent
In this note, we are interested in the generalization of all combinatorial aspects
of~\cite{NT}\footnote{We shall be using the results of~\cite[Section 5]{NT} that are omitted in its journal version~\cite{NT2}.},
excluding all shuffle algebra considerations, to the so-called ``weighted'' version.
To this end, we order the infinite alphabet $\CI=\{i^{(d)} \,|\, i\in I,d\in \BZ\}$ via
\begin{equation}\label{eqn:lex affine intro}
  i^{(d)} < j^{(e)} \qquad \Longleftrightarrow \qquad
  d/c_i > e/c_j  \quad \text{  or  } \quad d/c_i = e/c_j  \text{ and } i<j \,,
\end{equation}
for any fixed collection of ``weights'' $\{c_i\}_{i\in I}\in \BZ_{>0}^I$ (the case $c_i=1 \ \forall\, i$
recovers the setup of~\cite{NT}). This induces the lexicographic order on the loop words
$[i_1^{(d_1)} \dots\, i_k^{(d_k)}]$ with respect to which we may define the notion of \emph{\slaws}
by analogy with~\cite{LR}, which though requires some preliminary work similar to~\cite{NT}.
Then, there exists a one-to-one correspondence:
\begin{equation*}
  \ell \colon \Delta^+ \times \BZ \iso \Big\{\slaws\Big\} .
\end{equation*}
The lexicographic order on the right-hand side induces a convex order on the left-hand side,
with respect to which one can define elements
\begin{equation}\label{eqn:root vectors intro 2}
  e_{\ell(\alpha,d)} \in \UUp
\end{equation}
for all $(\alpha, d) \in \Delta^+ \times \BZ$. We have the following analogue of the PBW theorem:
\begin{equation}\label{eqn:pbw intro loop}
  \UUp \ = \ \bigoplus^{k \in \BN}_{\ell_1 \geq \dots \geq \ell_k \text{ standard Lyndon loop words}}
  \BQ(q)\cdot e_{\ell_1} \dots e_{\ell_k} \,.
\end{equation}
There are also analogues of the constructions above with $+ \leftrightarrow -$ and $e \leftrightarrow f$.

By analogy with the results of~\cite{L,R2}, the total order on $\Delta^+ \times \BZ$ given by
\begin{equation}\label{eqn:induces affine}
  (\alpha,d) < (\beta,e) \quad \Longleftrightarrow \quad
  \ell(\alpha,-d) < \ell(\beta,-e) \ \ \text{ lexicographically}
\end{equation}
is convex, cf.~Proposition \ref{prop:convex loop}. In fact, this order comes from a certain reduced word
in the affine Weyl group associated to $\fg$ (= the Coxeter group associated to $\widehat{\fg}$),
in accordance with Theorem~\ref{thm:weyl to lyndon}. Therefore, the root vectors~\eqref{eqn:root vectors intro 2}
exactly match (up to constants) the classical construction of~\cite{B,D,Lu}, once we pass it through the
``affine to loop'' isomorphism of Theorem~\ref{thm:two presentations}.


\subsection{Outline}\label{ssec:outline}
\

\noindent
The structure of the present paper is the following:
\begin{itemize}[leftmargin=0.5cm]

\item[$\bullet$]
In Section~\ref{sec:lie}, we recall the notion of (standard) Lyndon words, their basic properties,
and the application to simple Lie algebras through the bijection~\eqref{eqn:1-to-1 intro}.

\item[$\bullet$]
In Section~\ref{sec:loop algebra}, we study the loop Lie algebras $L\fg$ and generalize the results of the
previous Section to the loop setup with the order given by~\eqref{eqn:lex affine intro}. The key new ingredient,
in comparison to~\cite{NT}, is played by Theorem~\ref{thm:ExpRule} and Proposition~\ref{prop:firstletter}.

\item[$\bullet$]
In Section~\ref{sec:weyl lyndon}, we show that the order~\eqref{eqn:induces affine} on $\Delta^+ \times \BZ$
corresponds to a certain reduced decomposition in the extended affine Weyl group of $\fg$. We further refine
this result in Propositions~\ref{prop:terminal-segment}--\ref{prop:terminal-segment-2}.

\item[$\bullet$]
In Section~\ref{sec:quantum}, we construct PBW-type bases~\eqref{eqn:pbw intro loop} of the quantum loop
algebra $\UU$ by adapting the arguments of~\cite{NT} with the help of Proposition~\ref{prop:terminal-segment-2}.

\item[$\bullet$]
In Section~\ref{sec:generalization}, we adapt most of our results to more general orders~\eqref{eqn:lex generalized} on $\CI$.

\item[$\bullet$]
In Appendix~\ref{sec:appendix}, we provide a link to the C++ code and explain how it inductively computes
standard Lyndon loop words in all types, and present some examples.

\end{itemize}


\subsection{Acknowledgement}\label{sub:Achowledgement}
\

This note is written as a part of the project under the Yulia's Dream initiative, a subdivision of
the MIT PRIMES program aimed at Ukrainian high-school students.

\noindent
We are very grateful to the organizers Pavel Etingof, Slava Gerovich, and Dmytro Matvieievskyi
for giving us the opportunity to participate in the Yulia's Dream program.
S.K., N.K., K.M.\ are thankful to their parents for their help and support.

A.T.\ is deeply indebted to Andrei Negu\c{t} for sharing his invaluable insights, teaching the Lyndon word's theory,
and stimulating discussions through the entire project. A.T.\ is grateful to INdAM-GNSAGA and FAR UNIMORE project
CUP-E93C2300204000 for the support and wonderful working conditions during his visit to Italy in June 2024,
where the final version of this note was prepared. We are very grateful to the referees for their useful suggestions
that improved the exposition.

The work of A.T.\ was partially supported by an NSF Grant DMS-$2302661$.


\medskip

\section{Combinatorial approach to Lie algebras}\label{sec:lie}

In this Section, we recall the results of~\cite{LR} and~\cite{L} that provide a combinatorial construction
of an important basis of finitely generated Lie algebras, with the main application to the maximal nilpotent
subalgebra of a simple Lie algebra.


\subsection{Lyndon words}\label{sub:finite words}
\

Let $I$ be a finite ordered alphabet, and let $I^*$ be the set of all finite length words in the alphabet $I$.
For $u=[i_1 \dots i_k]\in I^*$, we define its \emph{length} by $|u|=k$. We introduce the \emph{lexicographic order}
on $I^*$ in a standard way:
\begin{equation*}
  [i_1 \dots i_k] < [j_1 \dots j_l] \quad \text{if }\
  \begin{cases}
    i_1=j_1, \dots, i_a=j_a, i_{a+1} < j_{a+1} \text{ for some } a \geq 0 \\
      \text{\ or} \\
    i_1=j_1, \dots, i_k=j_k \text{ and } k < l
  \end{cases} \,.
\end{equation*}

\begin{definition}\label{def:lyndon}
A word $\ell=[i_1\dots i_k]$ is called \underline{Lyndon} if it is smaller than all of its cyclic permutations:
\begin{equation*}
  [i_1 \dots i_{a-1} i_a \dots i_k] < [i_a \dots i_k i_1 \dots i_{a-1}] \qquad \forall\, a \in \{2,\dots,k\} \,.
\end{equation*}
\end{definition}

For a word $w = [i_1 \dots i_k]\in I^*$, the subwords
\begin{equation*}
  w_{a|} =  [i_1 \dots i_a] \qquad \text{and} \qquad w_{|a} = [i_{k-a+1} \dots i_k]
\end{equation*}
with $0\leq a\leq k$ will be called a \emph{prefix} and a \emph{suffix} of $w$, respectively.
We call such a prefix or a suffix \emph{proper} if $0<a<k$. It is straightforward to show that
Definition~\ref{def:lyndon} is equivalent to the following one:

\begin{definition}\label{def:lyndon-2}
A word $w$ is \underline{Lyndon} if it is smaller than all of its proper suffixes:
\begin{equation*}
  w < w_{|a} \qquad \forall\ 0<a<|w| \,.
\end{equation*}
\end{definition}

The following simple result is well-known:

\begin{lemma}\label{lemma:lyndon}
If $\ell_1 < \ell_2$ are Lyndon, then $\ell_1\ell_2$ is also Lyndon, and so $\ell_1 \ell_2 < \ell_2 \ell_1$.
\end{lemma}

We recall the following two basic facts from the theory of Lyndon words:

\begin{proposition}\label{prop:costandard factorization}(\cite[Proposition 5.1.3]{Lo})
Any Lyndon word $\ell$ has a factorization
\begin{equation}\label{eqn:costandard factorization}
  \ell = \ell_1 \ell_2
\end{equation}
defined by the property that $\ell_2$ is the longest proper suffix of $\ell$ which is also a Lyndon word.
Under these circumstances, $\ell_1$ is also a Lyndon word.
\end{proposition}

The factorization~\eqref{eqn:costandard factorization} is called the \emph{costandard factorization} of a Lyndon word.

\begin{proposition}\label{prop:canonical factorization}(\cite[Proposition 5.1.5]{Lo})
Any word $w$ has a unique factorization
\begin{equation}\label{eqn:canonical factorization}
  w = \ell_1 \dots \ell_k \,,
\end{equation}
where $\ell_1 \geq \dots \geq \ell_k$ are all Lyndon words.
\end{proposition}

The factorization~\eqref{eqn:canonical factorization} is called the \emph{canonical factorization} of a word.


\subsection{Standard Lyndon words}\label{sub:standard words}
\

Let $\fa$ be a Lie algebra generated by a finite set $\{e_i\}_{i\in I}$ labelled by the alphabet~$I$.

\begin{definition}\label{def:bracketing lyndon}
The standard bracketing of a Lyndon word $\ell$ is given inductively~by:
\begin{itemize}[leftmargin=0.7cm]

\item[$\bullet$]
$e_{[i]}=e_i\in \fa$ for $i \in I$,

\item[$\bullet$]
$e_{\ell} = [e_{\ell_1}, e_{\ell_2}]\in \fa$, where $\ell=\ell_1\ell_2$ is
the costandard factorization~\eqref{eqn:costandard factorization}.

\end{itemize}
\end{definition}

The major importance of this definition is due to the following result of Lyndon:

\begin{theorem}\label{thm:Lyndon theorem}(\cite[Theorem 5.3.1]{Lo})
If $\fa$ is a free Lie algebra in the generators $\{e_i\}_{i\in I}$, then the set
$\big\{e_{\ell} \,|\, \ell\mathrm{-Lyndon\ word}\big\}$ provides a basis of $\fa$.
\end{theorem}

It is natural to ask if Theorem~\ref{thm:Lyndon theorem} admits a generalization to Lie algebras $\fa$ generated
by $\{e_i\}_{i\in I}$ but with some defining relations. The answer was provided a few decades later in~\cite{LR}.
To state the result, define ${_we},e_w\in U(\fa)$ for any $w\in I^*$:
\begin{itemize}[leftmargin=0.7cm]

\item[$\bullet$]
For a word $w = [i_1 \dots i_k]\in I^*$, we set
\begin{equation}\label{eqn:the word}
  _we = e_{i_1} \dots e_{i_k} \in U(\fa)
\end{equation}

\item[$\bullet$]
For a word $w\in I^*$ with the canonical factorization $w=\ell_1 \dots \ell_k$
of~\eqref{eqn:canonical factorization}, we~set
\begin{equation}\label{eqn:bracket.word}
  e_w = e_{\ell_1} \dots e_{\ell_k} \in U(\fa) \,.
\end{equation}

\end{itemize}
It is well-known that the elements~\eqref{eqn:the word} and~\eqref{eqn:bracket.word} are connected
by the following triangularity property:
\begin{equation}\label{eqn:upper}
  e_w=\sum_{v \geq w} c^v_w \cdot {_ve} \qquad \mathrm{with} \quad
  c^v_w\in \BZ \quad \mathrm{and}\quad c_w^w = 1 \,.
\end{equation}

The following definition is due to \cite{LR}:

\begin{definition}\label{def:standard}
(a) A word $w$ is called \underline{standard} if $_we$ cannot be expressed as
a linear combination of $_ve$ for various $v>w$.

\medskip
\noindent
(b) A Lyndon word $\ell$ is called \underline{standard Lyndon} if $e_{\ell}$ cannot be expressed as
a linear combination of $e_m$ for various Lyndon words $m>\ell$.
\end{definition}

The following result is nontrivial and justifies the above terminology:

\begin{proposition}\label{prop:standard}(\cite{LR})
A Lyndon word is standard iff it is standard Lyndon.
\end{proposition}

The major importance of this definition is due to the following result:

\begin{theorem}\label{thm:standard Lyndon theorem}(\cite{LR})
For any Lie algebra $\fa$ generated by a finite collection $\{e_i\}_{i\in I}$, the set
$\big\{e_{\ell} \,|\, \ell\mathrm{-standard\ Lyndon\ word}\big\}$ provides a basis of $\fa$.
\end{theorem}

We also have the following simple properties of standard words:

\begin{proposition}\label{prop:factor standard}(\cite{LR})
(a) Any subword of a standard word is standard.

\medskip
\noindent
(b) A word $w$ is standard iff it can be written (uniquely) as $w=\ell_1 \dots \ell_k$,
where $\ell_1\geq \dots \geq \ell_k$ are standard Lyndon words.
\end{proposition}

Thus, combining the classical Poincar\'{e}–Birkhoff–Witt theorem for $U(\fa)$ with
Theorem~\ref{thm:standard Lyndon theorem}, Proposition~\ref{prop:factor standard}, and
the triangularity property \eqref{eqn:upper}, we obtain the following PBW-type theorem:
\begin{equation}\label{eqn:pbw lie}
\begin{split}
  & U(\fa) \ =   \bigoplus^{k\in \BN}_{\ell_1 \geq \dots \geq \ell_k \text{ standard Lyndon words}}
  \BQ \cdot e_{\ell_1} \dots e_{\ell_k} \, = \\
  & \qquad \qquad \qquad
  \bigoplus_{w \text{--standard words}} \BQ \cdot e_w \ =
  \bigoplus_{w \text{--standard words}} \BQ \cdot \, _we \,.
\end{split}
\end{equation}


\subsection{Application to simple Lie algebras}\label{sub:LR-bijection}
\

Let $\fg$ be a simple Lie algebra with the root system $\Delta=\Delta^+ \sqcup \Delta^-$. Let
$\{\alpha_i\}_{i\in I}\subset \Delta^+$ be the simple roots, and $Q=\bigoplus_{i\in I} \BZ\alpha_i$
be the root lattice. We endow $Q$ with the symmetric pairing $(\cdot,\cdot)\colon Q\otimes Q\to \BZ$
so that the Cartan matrix $(a_{ij})_{i,j\in I}$ and the symmetrized Cartan matrix $(d_{ij})_{i,j\in I}$
of $\fg$ are given by
\begin{equation*}
  a_{ij} = \frac {2(\alpha_i,\alpha_j)}{(\alpha_i,\alpha_i)}
  \qquad  \mathrm{and} \qquad d_{ij} = (\alpha_i,\alpha_j) \,.
\end{equation*}
Explicitly, $\fg$ is generated by $\{e_i, f_i, h_i\}_{i\in I}$ subject to the following defining relations:
\begin{equation}\label{eqn:rel 1 finite lie}
  \underbrace{[e_i,[e_i,\cdots,[e_i,e_j]\cdots]]}_{1-a_{ij} \text{ Lie brackets}}\, =\, 0 \qquad \text{if }i \neq j \,,
\end{equation}
\begin{equation}\label{eqn:rel 2 finite lie}
  [h_i,e_j] = d_{ij} e_j, \qquad \qquad [h_i,h_j] = 0 \,,
\end{equation}
as well as the opposite relations with $e$'s replaced by $f$'s, and finally the relation:
\begin{equation}\label{eqn:rel 3 finite lie}
  [e_i, f_j] = \delta_{ij} h_i \,.
\end{equation}
We will consider the triangular decomposition~\eqref{eqn:decomp intro 1}, where $\fn^+$, $\fh$, $\fn^-$
are the Lie subalgebras of $\fg$ generated by the $e_i$, $h_i$, $f_i$, respectively. We write $Q^+ \subset Q$
for the monoid generated by $\{\alpha_i\}_{i\in I}$. The Lie algebra $\fg$ is naturally $Q$-graded via
  $$ \deg e_i = \alpha_i \,,\qquad \deg h_i = 0 \,, \qquad \deg f_i = -\alpha_i \,. $$
The Lie algebra $\fg$ admits the standard \emph{root space decomposition}:
\begin{equation}\label{eq:root.decomp}
  \fg=\fh \oplus \bigoplus_{\alpha \in \Delta} \fg_{\alpha}
\end{equation}
with $\dim \fg_{\alpha}=1$ for all $\alpha\in \Delta$. We pick \emph{root vectors} $e_\alpha\in \fg_\alpha$
so that $\fg_\alpha=\BQ\cdot e_\alpha$. Thus, the Lie subalgebra $\fn^+$ decomposes into
$\fn^+=\bigoplus_{\alpha \in \Delta^+} \fg_{\alpha}$ and is $Q^+$-graded. Explicitly, $\fn^+$ is generated
by $\{e_i\}_{i\in I}$ subject to the classical \emph{Serre} relations~\eqref{eqn:rel 1 finite lie}.

Fix any order on the set $I$. According to Theorem~\ref{thm:standard Lyndon theorem}, $\fn^+$ has a basis
consisting of the $e_{\ell}$'s, as $\ell$ ranges over all standard Lyndon words. Evoking the above $Q^+$-grading
of the Lie algebra $\fn^+$, it is natural to define the grading of words via
\begin{equation*}
  \deg\, [i_1 \dots i_k] = \alpha_{i_1} + \dots + \alpha_{i_k} \in Q^+ \,.
\end{equation*}
Due to the decomposition~\eqref{eq:root.decomp} and the fact that the root vectors
$\{e_\alpha\}_{\alpha\in \Delta^+} \subset \fn^+$ all live in distinct degrees $\alpha \in Q^+$,
we conclude that there exists a bijection~\eqref{eqn:1-to-1 intro}:
\begin{equation*}
  \ell \colon \Delta^+ \iso \Big\{\text{standard Lyndon words}\Big\}
\end{equation*}
such that $\deg \ell(\alpha) = \alpha$ for all $\alpha \in \Delta^+$,
which we call the \emph{Lalonde-Ram's bijection}.


\subsection{Results of Leclerc}
\

The Lalonde-Ram's bijection~\eqref{eqn:1-to-1 intro} was described explicitly in~\cite{L}.
To state the result, we recall that for a root $\alpha=\sum_{i\in I} k_i\alpha_i\in \Delta^+$,
its \emph{height} is $\hgt(\alpha)=\sum_i k_i$.

\begin{proposition}\label{prop:Leclerc algorithm}(\cite[Proposition 25]{L})
The bijection $\ell$ is inductively given by:
\begin{itemize}[leftmargin=0.7cm]

\item[$\bullet$]
for simple roots, we have $\ell(\alpha_i)=[i]$,

\item[$\bullet$]
for other positive roots, we have the following \underline{Leclerc's algorithm}:
\begin{equation}\label{eqn:inductively}
  \ell(\alpha) =
  \max\left\{ \ell(\gamma_1)\ell(\gamma_2) \,\Big|\,
               \alpha=\gamma_1+\gamma_2 \,,\, \gamma_1,\gamma_2\in \Delta^+ \,,\, \ell(\gamma_1) < \ell(\gamma_2) \right\}.
\end{equation}
\end{itemize}
\end{proposition}

\noindent
The formula~\eqref{eqn:inductively} recovers $\ell(\alpha)$ once we know $\ell(\gamma)$ for all
$\{\gamma\in \Delta^+ \,|\, \mathrm{ht}(\gamma)<\mathrm{ht}(\alpha)\}$.

We shall also need one more important property of $\ell$. To the end, let us recall:

\begin{definition}\label{def:convex}
A total order on the set of positive roots $\Delta^+$ is \underline{convex} if:
\begin{equation*}
  \alpha < \alpha+\beta < \beta
\end{equation*}
for all $\alpha < \beta \in \Delta^+$ such that $\alpha+\beta$ is also a root.
\end{definition}

\begin{remark}\label{rem:Papi's bijection}
It is well-known (\cite{P}) that convex orders on $\Delta^+$ are in bijection with
the reduced decompositions of the longest element $w_0\in W$ in the Weyl group of $\fg$.
\end{remark}

The following result is~\cite[Proposition 26]{L}, where it was attributed to the preprint of Rosso~\cite{R2}
(a detailed proof can be found in~\cite[Proposition 2.34]{NT}):

\begin{proposition}\label{prop:finite convexity}
Consider the order on $\Delta^+$ induced from the lexicographic order on standard Lyndon words:
\begin{equation*}
  \alpha < \beta \quad \Longleftrightarrow \quad \ell(\alpha) < \ell(\beta)  \ \ \mathrm{ lexicographically} \,.
\end{equation*}
This order is convex.
\end{proposition}


\medskip

\section{Standard Lyndon loop words}\label{sec:loop algebra}

We will now extend the description above to the Lie algebra of loops into $\fg$:
\begin{equation*}
  L\fg = \fg[t,t^{-1}]=\fg \otimes_{\BQ} \BQ [t,t^{-1}]
\end{equation*}
with the Lie bracket given simply by
\begin{equation*}
  [x \otimes t^m, y \otimes t^n] = [x,y] \otimes t^{m+n}
  \quad \mathrm{for\ any} \quad  x,y \in \fg \,,\, m,n \in \BZ \,.
\end{equation*}
The triangular decomposition~\eqref{eqn:decomp intro 1} extends to a similar decomposition at
the loop level $L\fg=L\fn^+ \oplus L\fh \oplus L\fn^-$, and our goal is to describe $L\fn^+$
along the lines of the previous Section. To this end, we think of $L\fn^+$ as being generated
by $e_i^{(d)}=\bare_i\otimes t^d$ for all $i\in I, d\in \BZ$. Associate to $e_i^{(d)}$
the letter $i^{(d)}$, and call $d$ the \emph{exponent} of $i^{(d)}$.

We thus obtain the infinite alphabet $\CI=\{i^{(d)} \,|\, i\in I,d\in \BZ\}$ and
any word in these letters will be called a \emph{loop word}:
\begin{equation}\label{eqn:loop word}
  \left[i_1^{(d_1)} \dots\, i_k^{(d_k)}\right] .
\end{equation}
We shall now introduce a family of total orders on $\CI$, which will thus induce lexicographic orderings on
loop words~\eqref{eqn:loop word}. To this end, we fix a total order on $I$ and choose a tuple of positive integers
$\{c_i\}_{i\in I}\in \BZ_{>0}^I$ (we call $c_i$ the \emph{weight} of $i$). Following~\cite[Remark 3.15]{NT},
we shall compare the loop letters of $\CI$ via~\eqref{eqn:lex affine intro}:
\begin{equation*}
  i^{(d)} < j^{(e)} \qquad \Longleftrightarrow \qquad
  \frac{d}{c_i} > \frac{e}{c_j}  \quad \text{  or  } \quad
  \frac{d}{c_i} = \frac{e}{c_j} \text{ and } i<j \,.
\end{equation*}
Due to its importance, we shall call the ratio $d/c_i$ the \emph{relative exponent} of $i^{(d)}\in \CI$.
We also define the \emph{weighted height} of roots via:
\begin{equation}\label{eqn:weighted height}
  f(\alpha) = \sum_{i\in I} k_i\cdot c_i \quad \mathrm{for\ any} \quad
  \alpha = \sum_{i \in I} k_i \alpha_i \in \Delta^+ \,.
\end{equation}

All the results of Subsection~\ref{sub:finite words} continue to hold in the present setup, so we have a notion
of Lyndon loop words. Since $L\fn^+$ is $Q^+\times \BZ$-graded via $\deg e_i^{(d)} = (\bs_i,d)$, it makes sense
to extend this grading to loop words via
\begin{equation*}
  \deg \left[i_1^{(d_1)} \dots\, i_k^{(d_k)} \right] = (\bs_{i_1}+ \dots +\bs_{i_k}, d_1+ \dots +d_k) \,.
\end{equation*}
The obvious generalization of~\eqref{eqn:root vectors intro} is:
\begin{equation*}
  L\fn^+ = \bigoplus_{\alpha \in \Delta^+} \bigoplus_{d \in \BZ} \BQ \cdot e_\alpha^{(d)}
\end{equation*}
with $e^{(d)}_\alpha=e_\alpha\otimes t^d$ for all $\alpha\in \Delta^+, d\in \BZ$. We note that $L\fn^+$
still has one-dimensional $Q^+ \times \BZ$-graded pieces, which is essential for the treatment of~\cite{LR}
to carry through.

On the other hand, the definition of standard (Lyndon) loop words in the present setup is a non-trivial task
since the alphabet $\CI$ is infinite. Motivated by the treatment of~\cite{NT} in the case when all $c_i=1$,
we shall likewise consider a filtration by finitely generated Lie algebras $L^{(s)}\fn^+$ of~\eqref{eqn:loop filtration},
corresponding to the finite alphabets
\begin{equation}\label{eqn:finite loop alphabet}
  \CI^{(s)}=\left\{i^{(d)} \,\Big|\, i\in I, -s\cdot c_i\leq d\leq s\cdot c_i\right\}
  \qquad \forall\, s\in \BN \,.
\end{equation}
We will establish some basic properties of the corresponding standard Lyndon loop words for $L^{(s)}\fn^+$
which ultimately imply that the notion of a ``standard Lyndon loop word'' does not actually depend on the particular
$L^{(s)}\fn^+$ with respect to which it is defined. We shall thus obtain the loop analogue~\eqref{eqn:associated word loop}
of the bijection~\eqref{eqn:1-to-1 intro}.


\subsection{Filtration and basic properties}\label{sub:affine standard}
\

We now wish to extend Definition~\ref{def:standard} in order to obtain a notion of standard (Lyndon)
loop words, but here we must be careful as the alphabet $\CI$ is infinite. In particular, the key assumption
``for any word $v$, there are only finitely many words $u$ of the same length and $>v$ in the lexicographic order''
of~\cite[\S2]{LR} clearly fails. To deal with this issue, we consider the increasing filtration:
\begin{equation*}
  L\fn^+ = \bigcup_{s=0}^\infty L^{(s)}\fn^+
\end{equation*}
defined with respect to the finite-dimensional Lie subalgebras (see notation~\eqref{eqn:weighted height}):
\begin{equation}\label{eqn:loop filtration}
  L\fn^+ \supset L^{(s)}\fn^+ =
  \bigoplus_{\alpha \in \Delta^+} \bigoplus_{d = -s\cdot f(\alpha)}^{s\cdot f(\alpha)} \BQ \cdot e_\alpha^{(d)}
  \qquad \forall\, s\in \BN \,.
\end{equation}
As a Lie algebra, $L^{(s)}\fn^+$ is generated by $\{e_i^{(d)} \,|\, i\in I, |d| \leq s\cdot c_i\}$. We may thus
apply Definition~\ref{def:standard} to yield a notion of standard (Lyndon) loop words with respect to the
finite-dimensional Lie algebras $L^{(s)}\fn^+$, with the words made up only of $i^{(d)}\in \CI^{(s)}$.

The following result is proved completely analogously to~\cite[Proposition 2.23]{NT}
(which in turn is an adaptation of the analogous results from~\cite{L}, cf.~\eqref{eqn:inductively}):

\begin{proposition}\label{prop:classification}
There exists a bijection:
\begin{equation}\label{eqn:bijection lyndon}
  \ell \colon
  \Big\{(\alpha,d) \in \Delta^+ \times \BZ \, \Big|\, |d| \leq s\cdot f(\alpha) \Big\} \iso
  \Big\{\substack{\text{standard Lyndon loop}\\ \text{words for } L^{(s)}\fn^+} \Big\} ,
\end{equation}
determined by $\ell(\alpha_i,d)=\left[i^{(d)}\right]$ and the following (generalized) \underline{Leclerc's algorithm}:
\begin{equation}\label{eqn:property lyndon}
  \ell(\alpha,d) \ =
  \mathop{\mathop{\max_{(\gamma_1,d_1)+(\gamma_2,d_2) = (\alpha,d)}}_{\gamma_k \in \Delta^+, \ |d_k| \leq s\cdot f(\gamma_k)}}_{\ell(\gamma_1,d_1) < \ell(\gamma_2,d_2)}
  \Big\{ \text{concatenation } \ell(\gamma_1,d_1)\ell(\gamma_2,d_2) \Big\} .
\end{equation}
\end{proposition}

Since standard Lyndon loop words give rise to bases of the finite-dimensional Lie algebras $L^{(s)} \fn^+$,
then the analogue of property \eqref{eqn:pbw lie} gives us:
\begin{multline}\label{eqn:pbw lie loop filtration}
  U(L^{(s)}\fn^+) \ =
  \mathop{\mathop{\bigoplus^{k\in \BN}}_{\ell_1 \geq \dots \geq \ell_k \text{ standard Lyndon loop words}}}_
  {\text{with all relative exponents in } [-s,s]} \BQ \cdot e_{\ell_1} \dots e_{\ell_k} \, = \\
  \mathop{\mathop{\bigoplus}_{w \text{--standard loop words with}}}_{\text{all relative exponents in } [-s,s]} \BQ \cdot e_w \ =
  \mathop{\mathop{\bigoplus}_{w \text{--standard loop words with}}}_{\text{all relative exponents in } [-s,s]} \BQ \cdot\, _we \,.
\end{multline}

We shall next establish some properties of the bijection~\eqref{eqn:bijection lyndon}.
We start with the following \emph{monotonicity} property:

\begin{proposition}\label{prop:l1}
Fix $s\in \BZ_{>0}$. Then for any positive root $\alpha\in \Delta^+$ and any integer
$d\in [-s\cdot f(\alpha)+1,s\cdot f(\alpha)]$, the bijection~\eqref{eqn:bijection lyndon}
satisfies the following inequality:
\begin{equation}\label{eqn:inequality lyndon}
  \ell(\alpha,d) < \ell(\alpha,d-1) \,.
\end{equation}
\end{proposition}

\begin{proof}
The proof is completely analogous to that of~\cite[Proposition~2.25]{NT}.
\end{proof}


\subsection{Exponent tightness}\label{sub:exp tightness}
\

While many properties of the bijection~\eqref{eqn:bijection lyndon} can be established very similarly to the
special case (when $c_i=1$ for all $i$) of~\cite{NT}, the naive generalization of~\cite[Proposition~2.26]{NT}
shall not suffice. We discuss the key upgrades in this~Subsection.

We start with the following definition:

\begin{definition}\label{exponent-tight}
A loop word $ w=\left[i_1^{(d_{1})} \dots \, i_n^{(d_{n})}\right]$ is called \underline{exponent-tight} if
\begin{equation}\label{eq:assumption}
  i_k^{(d_{k})} \geq i_{r}^{(d_{r} + 1)} \qquad \mathrm{for\ all} \quad 1\leq k, r \leq n \,.
\end{equation}
\end{definition}

When $w$ is a Lyndon loop word, it clearly suffices to verify~\eqref{eq:assumption} only for $k = 1$.
The following is the main result of this Subsection:

\begin{theorem}\label{thm:ExpRule}
For any root $\alpha\in \Delta^+$ and any integer $d \in \{-s\cdot f(\alpha), \ldots, s\cdot f(\alpha)\}$,
the standard Lyndon loop word $\ell(\alpha, d)$ is exponent-tight.
\end{theorem}

The proof of this result relies on Lemma~\ref{lem:Uniquexponenttight} and Proposition~\ref{prop:firstletter} proved below.
In what follows, we write $i^{(d)}\in w$ to denote that $w$ contains the letter $i^{(d)}\in \CI$. If a loop word $w$ has
a $Q\times \BZ$-degree $\deg w=(\alpha,d)$, then we will use the notation
\begin{equation}\label{eq:hor and vert words}
  \hdeg w = \alpha \qquad \text{and} \qquad \vdeg w = d \,,
\end{equation}
and call these two notions the \emph{horizontal} and the \emph{vertical} degree, respectively.

\begin{lemma}\label{lem:Uniquexponenttight}
Any two exponent-tight loop words $v$ and $w$ of the same $Q\times \BZ$-degree contain the same multisets of letters.
\end{lemma}

\begin{proof}
First, let us show that if $i^{(k)}\in w$ then also $i^{(k)}\in v$. Assuming the contradiction,
we must have $i^{(k')}\in v$ for some $k'\ne k$, as $\hdeg v=\hdeg w$. Without loss of generality,
we may assume that $k' \geq k+1$, so that $i^{(k')}\leq i^{(k+1)}$. As $\vdeg v = \vdeg w$,
there are two letters $j^{(t)} \in w$ and $j^{(t')} \in v$, such that $t' \leq t-1$, so that
$j^{(t)}\leq j^{(t'+1)}$. Since both words $v$ and $w$ are exponent-tight, we also have
  $$ i^{(k+1)}\leq j^{(t)} \qquad \mathrm{and} \qquad j^{(t'+1)}\leq i^{(k')} \,. $$
Combining the above inequalities, we obtain:
  $$ j^{(t)}\leq j^{(t'+1)}\leq i^{(k')}\leq i^{(k+1)}\leq j^{(t)} \,, $$
so that $j^{(t)}=j^{(t'+1)}=i^{(k')}=i^{(k+1)}$. Hence $i^{(k)}=j^{(t')}\in v$, a contradiction.

Thus any letter of $w$ is contained in $v$ and vice-versa. It remains to show that multiplicities
of all letters in $w$ and $v$ are the same. Since $\hdeg v= \hdeg w$, the sum of all multiplicities
of $i^{(\bullet)}\in w$ is the same as that of $i^{(\bullet)}\in v$ for any $i\in I$. Thus, the claim
is obvious if both $w$ and $v$ contain $i^{(k)}$ and no other $i^{(k')}$ for $k'\ne k$. Assume now that $w$
(and hence also $v$) contains $i^{(k)}, i^{(k')}$ for $k'>k$. Then $k'=k+1$, due to $i^{(k')}\geq i^{(k+1)}$.
In this case, we may not have $j^{(t)},j^{(t+1)}\in w$ for any $j\ne i$ and $t\in \BZ$. Otherwise we would
have $i^{(k+1)}\geq j^{(t+1)}\geq i^{(k+1)}$, due to exponent-tightness, and so $j^{(t)}=i^{(k)}$,
a contradiction with $j\ne i$. Thus, for any $j\ne i$, there is only one value of exponent such that
$j^{(\bullet)}$ is contained in $w$ (and hence in $v$). As $\deg v=\deg w$, we thus also conclude that
multiplicities of $i^{(k)}, i^{(k+1)}$ in $w$ and $v$ are the same.
\end{proof}

\begin{proposition}\label{prop:firstletter}
Let $v=[i_1^{(d_1)} \dots \, i_m^{(d_m)}]$ and $w=[j_1^{(t_1)} \dots \, j_m^{(t_m)}]$ be two exponent-tight
loop words such that $\hdeg w=\hdeg v$, $\vdeg w=\vdeg v + 1$, and $j_1^{(t_1)}\leq j_r^{(t_r)}$ for all $r$. Then:

\medskip
\noindent
(a) The first letter $j_1^{(t_1)}$ of the loop word $w$ equals $\max_{1\leq a\leq m}\, \{i_a^{(d_{a} + 1)}\}$;

\medskip
\noindent
(b) The multisets of the other letters coincide:
$\{i_a^{(d_a)}\}_{a=1}^m-\{j_1^{(t_1-1)}\}=\{j_a^{(t_a)}\}_{a=2}^m$.
\end{proposition}

\begin{proof}
Let $i_r^{(d_r+1)}=\max_{1\leq a\leq m}\, \{i_a^{(d_{a} + 1)}\}$. Since $v$ is exponent-tight, so is
any loop word $u$ formed by the letters $\{i_a^{(d_a)}\}_{a\ne r}\cup \{i_r^{(d_r+1)}\}$ (a loop word
is exponent-tight iff any loop word formed by the same multiset of letters is exponent-tight). But then
$w$ and $u$ must have the same multisets of letters, according to Lemma~\ref{lem:Uniquexponenttight}.
Since the loop word $u$ satisfies (b), $v$ is exponent-tight and $w$ starts with its smallest letter,
we obtain both properties (a) and (b).
\end{proof}

\begin{remark}\label{rem:first-letter-minus}
Following the setup of Proposition~\ref{prop:firstletter}, one may vice-versa express the multiset
of letters of $v$ through the one for $w$: $\{i_a^{(d_a)}\}_{a=1}^m=\{j_1^{(t_1-1)}\}\cup \{j_a^{(t_a)}\}_{a=2}^m$.
\end{remark}

Now we are ready to present the proof of  Theorem~\ref{thm:ExpRule}.

\begin{proof}[Proof of Theorem~\ref{thm:ExpRule}]
The proof proceeds by induction on the height $n=\hgt(\alpha)$.

The base case of the induction is $n=2$. Let $\ell(\alpha, d) = [i_1^{(d_{1})}i_2^{(d_{2})}]$,
where $i_1^{(d_{1})}<i_2^{(d_{2})}$ and $i_1\neq i_2$. We claim that
$i_1^{(d_{1} - 1)} > i_2^{(d_{2} + 1)}$, as otherwise we would get
  $\ell(\alpha,d)=[i_1^{(d_{1})}i_2^{(d_{2})}] < [i_1^{(d_{1} - 1)}i_2^{(d_{2} + 1)}]$,
a contradiction with Leclerc's algorithm~\eqref{eqn:property lyndon}. But then,
invoking~\eqref{eqn:property lyndon}, we obtain $\ell(\alpha, d) \geq i_2^{(d_{2} + 1)}i_1^{(d_{1} - 1)}$.
This implies the desired inequality $i_1^{(d_{1})}\geq i_2^{(d_{2} + 1)}$, establishing the base of the induction.

Let us now prove the step of the induction, assuming the assertion holds for all roots of height $<n$.
If not, then for some root $\alpha\in \Delta^+$ of height $n$ and some $d\in \BZ$, we have
$\ell(\alpha,d)=[i_1^{(d_1)} \dots \, i_n^{(d_n)}]$ with $i_r^{(d_{r} + 1)} > i_1^{(d_{1})}$
for some $1<r\leq n$. Let us consider the costandard factorization of $\ell(\alpha,d)$:
\begin{align*}
  \ell(\alpha, d) = \ell(\gamma_1, k_1) \ell(\gamma_2, k_2) \,,
\end{align*}
where $\alpha=\gamma_1 + \gamma_2$, $d = k_1 + k_2$, $\ell(\gamma_1, k_1) < \ell(\gamma_2, k_2)$, and roots
$\gamma_1,\gamma_2$ have height~$<n$. By the induction hypothesis, $i_r^{(d_{r})}\notin \ell(\gamma_1, k_1)$,
so that $i_r^{(d_{r})}\in \ell(\gamma_2, k_2)$. Arguing as above, we claim that
$\ell(\gamma_1, k_1-1) > \ell(\gamma_2, k_2+1)$, as otherwise according to~\eqref{eqn:inequality lyndon} we would get
  $\ell(\alpha, d) = \ell(\gamma_1, k_1)\ell(\gamma_2, k_2) < \ell(\gamma_1, k_1 - 1) \ell(\gamma_2,  k_2 + 1)$,
a contradiction with~\eqref{eqn:property lyndon}. The inequality $\ell(\gamma_1, k_1-1) > \ell(\gamma_2, k_2+1)$
implies
\begin{align}\label{eq:label-it}
  \ell(\alpha, d) \geq \ell(\gamma_2,  k_2 + 1) \ell(\gamma_1,  k_1 - 1) \,,
\end{align}
due to~\eqref{eqn:property lyndon}. Since $\hgt(\gamma_2)<n$, both words $\ell(\gamma_2, k_2)$ and $\ell(\gamma_2, k_2+1)$
are exponent-tight by the induction hypothesis. Therefore, the first letter of $\ell(\gamma_2, k_2 + 1)$ is
$i_{t}^{(d_t+1)}=\max_{\hgt(\gamma_1)<a\leq n}\{i_a^{(d_{a} + 1)}\}$, due to Proposition~\ref{prop:firstletter}.
Note that $i_{t}^{(d_{t}+1)} \leq i_1^{(d_1)}$, according to~\eqref{eq:label-it}. Therefore, we get
$i_r^{(d_{r}+1)} \leq i_{t}^{(d_t+1)}\leq i_1^{(d_1)}$, a contradiction.
\end{proof}

\begin{remark}\label{rem:silly-generalization-NT}
Let us emphasize that applying directly the argument from the proof of \cite[Proposition 2.26]{NT},
one rather gets a weaker statement:
\begin{equation}\label{eq:silly-assumption}
  \ell(\alpha,d)=\left[ i^{(d_1)}_1 \dots \, i_n^{(d_n)} \right] \quad \mathrm{with} \quad
  \left\lfloor \frac{d}{f(\alpha)}\right\rfloor \leq \frac{d_r}{c_{i_r}}\leq \left\lceil \frac{d}{f(\alpha)}\right\rceil
  \quad \forall\, 1\leq r\leq n
\end{equation}
with $f(\alpha)$ defined in~\eqref{eqn:weighted height}. In particular, if $c_i=N>1$ for all $i\in I$
(thus the order on $\CI$ is the same as for $c_i=1$ and so $\ell(\alpha,d)$ are the same as in~\cite{NT}),
then~\eqref{eq:silly-assumption} only implies $|d_r-d_t|\leq N$, while Theorem~\ref{thm:ExpRule} implies
a much finer bound $|d_r-d_t|\leq 1$.
\end{remark}

The following is a simple corollary of Theorem~\ref{thm:ExpRule}:

\begin{corollary}\label{cor:first letter}
(a) For $\alpha\in \Delta^+,d>0$, the first letter of $\ell(\alpha,d)$ has exponent~$>0$.

\medskip
\noindent
(b) For $\alpha\in \Delta^+,d\leq 0$, the first letter of $\ell(\alpha,d)$ has exponent~$\leq 0$.
\end{corollary}

\begin{proof}
Let $\ell(\alpha,d)=[i_1^{(d_1)} \dots \, i_n^{(d_n)}]$. Then $i_1^{(d_1)}\leq i_r^{(d_r)}$ and so
$\frac{d_1}{c_{i_1}}\geq \frac{d_r}{c_{i_r}}$ for any $r$. Thus if $d_1\leq 0$, then $d_r\leq 0$ for any $r$,
and so $d=\sum_{r=1}^n d_r\leq 0$, implying part~(a).

To prove~(b), we note that $i_1^{(d_1)}\geq i_r^{(d_r+1)}$ for any $r$ by Theorem~\ref{thm:ExpRule},
thus $\frac{d_1}{c_{i_1}}\leq \frac{d_r+1}{c_{i_r}}$. If $d_1>0$, then $d_r\geq 0$ for all $r$, and so
$d=\sum_{r=1}^n d_r>0$, a contradiction.
\end{proof}


\subsection{Stabilization}
\

As an important consequence of Theorem~\ref{thm:ExpRule}, we obtain:

\begin{proposition}\label{prop:coherent}
Any loop word $w$ with relative exponents in $[-s,s]$ is standard (Lyndon) with respect to $L^{(s)}\fn^{+}$
iff it is standard (Lyndon) with respect to $L^{(s+1)}\fn^{+}$.
\end{proposition}

\begin{proof}
While the proof of~\cite[Proposition 2.28]{NT} can be directly generalized with the help of
Theorem~\ref{thm:ExpRule}, let us present a shorter argument. Consider loop words
\begin{equation*}
\begin{split}
  & \ell = \ell(\alpha, d) \text{ of~\eqref{eqn:bijection lyndon} with respect to $L^{(s)}\mathfrak{n}^+$} \,, \\
  & \ell' = \ell(\alpha, d) \text{ of~\eqref{eqn:bijection lyndon} with respect to $L^{(s+1)}\mathfrak{n}^+$} \,.
\end{split}
\end{equation*}
Combining~\eqref{eq:silly-assumption} with Theorem~\ref{thm:ExpRule} and Proposition~\ref{prop:firstletter},
we see that both words $\ell$ and $\ell'$ contain the same multisets of letters (all thus being elements of $\CI^{(s)}$).
Additionally, their standard bracketings $e_\ell$, $e_{\ell'}$ are both nonzero multiples of $e^{(d)}_\alpha$.
By the very definition of standard Lyndon loop words, this implies that $\ell=\ell'$.
\end{proof}

The above result implies that the notion of a ``standard Lyndon loop word'' does not depend on the particular
$L^{(s)}\fn^+$ with respect to which it is defined. We conclude that there exists a bijection:
\begin{equation}\label{eqn:associated word loop}
  \ell \colon \Delta^+\times \BZ \ \iso \ \Big\{\text{standard Lyndon loop words}\Big\}
\end{equation}
satisfying property~\eqref{eqn:property lyndon} with $s = \infty$ as well as Theorem~\ref{thm:ExpRule}
and Proposition~\ref{prop:firstletter}.


\subsection{Periodicity}
\

While $\ell$ of~\eqref{eqn:associated word loop} is a bijection between infinite sets, it is
actually determined by the values of $\ell$ only on a finite ``block'' of $\Delta^+\times \BZ$:
\begin{equation}\label{eqn:L-chunk}
  L= \Big\{(\alpha,d) \,\Big|\,\alpha\in \Delta^+, 0\leq d<f(\alpha) \Big\} ,
\end{equation}
cf.\ notation~\eqref{eqn:weighted height}. More precisely, we have the following \emph{periodicity} property:

\begin{proposition}\label{prop:peridocitiy}
For any $(\alpha,d) \in \Delta^+ \times \BZ$, the standard Lyndon loop word $\ell(\alpha,d+f(\alpha))$
is obtained from the standard Lyndon loop word $\ell(\alpha,d)$ by increasing all exponents of its
letters $i^{(\bullet)}$ by $c_i$ (that is, increasing all relative exponents~by~$1$).
\end{proposition}

\begin{proof}
Let $\Upsilon$ denote the aforementioned bijective map on the set of loop words:
\begin{equation}\label{eq:Upsilon}
  \Upsilon \colon \left[ i_1^{(d_1)} \ldots \, i_k^{(d_k)} \right] \mapsto \left[ i_1^{(d_1+c_{i_1})} \ldots \, i_k^{(d_k+c_{i_k})} \right] \,.
\end{equation}
Note that $u<v$ iff $\Upsilon(u)<\Upsilon(v)$ in accordance with~\eqref{eqn:lex affine intro}. Thus,~\eqref{eq:Upsilon}
preserves the property of a loop word being Lyndon. Likewise, if $\ell=\ell_1\ell_2$ is the costandard factorization
of $\ell$, then $\Upsilon(\ell)=\Upsilon(\ell_1)\Upsilon(\ell_2)$ is the costandard factorization of $\Upsilon(\ell)$.
This also implies that $e_{\Upsilon(\ell)}=\wt{\Upsilon}(e_{\ell})$, where $\wt{\Upsilon}$ is the Lie algebra isomorphism:
\begin{equation*}
   \wt{\Upsilon}\colon L\fn^+ \ \iso\ L \fn^+ \qquad \mathrm{given\ by} \qquad
   e_{\alpha}^{(d)} \mapsto e_\alpha^{(d+f(\alpha))} \,.
\end{equation*}
Hence,~\eqref{eq:Upsilon} also preserves the property of a Lyndon loop word being standard.
\end{proof}

Similarly to~\cite[Proposition 2.31]{NT}, we also note the following simple property:

\begin{proposition}\label{prop:horizontal}
The restriction of \eqref{eqn:associated word loop} to $\Delta^+ \times \{0\}$ matches \eqref{eqn:1-to-1 intro}.
\end{proposition}

\begin{proof}
This is simply the $s = 0$ case of Proposition \ref{prop:coherent}.
\end{proof}

Since $U(L\fn^+)$ is the direct limit as $s\rightarrow \infty$ of the $U(L^{(s)}\fn^+)$,
then \eqref{eqn:pbw lie loop filtration} implies:
\begin{equation}\label{eqn:pbw lie loop}
\begin{split}
  & U(L\fn^+) \ =
  \bigoplus^{k\in \BN}_{\ell_1 \geq \dots \geq \ell_k \text{ standard Lyndon loop words}}
    \BQ \cdot e_{\ell_1} \dots e_{\ell_k} \, = \\
  & \qquad \qquad \qquad
  \mathop{\bigoplus}_{w \text{--standard loop words}} \BQ \cdot e_w \ =
  \mathop{\bigoplus}_{w \text{--standard loop words}} \BQ \cdot\, _we \,.
\end{split}
\end{equation}


\subsection{Convexity and minimality}\label{sub:convex affLyndon}
\

We conclude this Section with a few fundamental properties of the total order on $\Delta^+ \times \BZ$
induced by transporting the lexicographic order on loop words via the bijection~\eqref{eqn:associated word loop}.
A straightforward generalization of~\cite[Proposition~2.34]{NT} establishes that this order is \emph{convex},
a notion that is a direct generalization of Definition~\ref{def:convex}:

\begin{proposition}\label{prop:convex loop}
For all $(\alpha,d) , (\beta,e), (\alpha+\beta, d+e) \in \Delta^+ \times \BZ$, we have:
\begin{equation*}
  \ell(\alpha,d) < \ell(\alpha+\beta,d+e) < \ell(\beta,e)
\end{equation*}
if $\ell(\alpha , d) < \ell(\beta,e)$.
\end{proposition}

This result admits the following natural generalization:

\begin{corollary}\label{cor:convex several}
Consider any $k,k' \geq 1$ and any
  $$ (\gamma_1,d_1), \dots, (\gamma_k,d_k) ,(\gamma'_1,d'_1), \dots, (\gamma'_{k'},d'_{k'})  \in \Delta^+ \times \BZ $$
such that $(\gamma_1,d_1)+ \dots + (\gamma_k,d_k) = (\gamma_1',d_1') + \dots + (\gamma_{k'}',d_{k'}')$.
Then we have:
\begin{equation*}
  \min \Big \{\ell(\gamma_1,d_1), \dots, \ell(\gamma_k,d_k) \Big\} \leq
  \max \Big \{\ell(\gamma'_1,d'_1), \dots, \ell(\gamma'_{k'},d'_{k'}) \Big\} .
\end{equation*}
\end{corollary}

\begin{proof}
The proof is completely analogous to that of~\cite[Corollary 2.37]{NT}.
\end{proof}

An important consequence of this Corollary is the following result, which will
play a crucial role in our proof of Theorem~\ref{thm:PBW quantum loop} below:

\begin{proposition}\label{prop:lyndon is minimal}
If $\ell_1 < \ell_2$ are standard Lyndon loop words such that $\ell_1\ell_2$ is also a standard Lyndon loop word,
then we cannot have:
$$
  \ell_1 < \ell_1' < \ell_2' < \ell_2
$$
for standard Lyndon loop words $\ell_1',\ell_2'$ such that $\deg \ell_1 + \deg \ell_2 = \deg \ell_1' + \deg \ell_2'$.
\end{proposition}

\begin{proof}
The proof is completely analogous to that of~\cite[Proposition 2.38]{NT}.
\end{proof}


\medskip

\section{Lyndon words and Weyl groups}\label{sec:weyl lyndon}

In this Section, we show that the order~\eqref{eqn:induces affine} on $\Delta^+ \times \BZ$
induced by \eqref{eqn:associated word loop} is related to the construction of~\cite{P1,P2} applied to a
reduced decomposition of a translation element in the extended affine Weyl group encoding the weights $c_i$.


\subsection{Affine Lie algebras}\label{sub:affine Weyl}
\

In this Subsection, we recall the next simplest class of Kac-Moody Lie algebras after the simple ones,
the affine Lie algebras. Let $\fg$ be a simple finite-dimensional Lie algebra, $\{\alpha_i\}_{i\in I}$ be
the simple roots, and $\theta\in \Delta^+$ be the highest root. The \emph{labels} of the Dynkin diagram
of $\fg$ are the positive integers $\{\theta_i\}_{i\in I}$ such that
\begin{equation}\label{eqn:labels}
  \theta = \sum_{i\in I} \theta_i \alpha_i \,.
\end{equation}
We define $\wI = I \sqcup \{0\}$. Consider the affine root lattice $\wQ$ with the generators
$\{\alpha_i\}_{i\in \wI}$ which admits a natural identification
\begin{equation}\label{eqn:finite vs affine roots}
  \wQ\ \iso\ Q \times \BZ \qquad \mathrm{with} \qquad
  \alpha_i \mapsto (\alpha_i,0) \ \ \forall\, i\in I, \qquad \alpha_0 \mapsto (-\theta,1) \,.
\end{equation}
We endow $\wQ$ with the symmetric pairing defined by:
\begin{equation*}
  \big((\alpha,n),(\beta,m)\big)=(\alpha,\beta) \qquad \forall\ \alpha,\beta\in Q \,,\, n,m \in \BZ.
\end{equation*}
As opposed from the non-degenerate pairing on $Q$ itself, the pairing on affine type root systems
has a one-dimensional kernel, which is spanned by the minimal \emph{imaginary root}
$\delta=\alpha_0+\theta=(0,1) \in Q \times \BZ$. This implies the fact that:
\begin{equation}\label{eq:kernel-parity}
  (\alpha_0 + \theta, - ) = 0 \quad \Longleftrightarrow \quad  d_{0j} + \sum_{i \in I} \theta_i d_{ij} = 0
  \qquad \forall\, j\in I \,,
\end{equation}
where $\{d_{ij}\}_{i,j\in \wI}$ is the symmetrized affine Cartan matrix. Let $(a_{ij})_{i,j\in \wI}$
be the affine Cartan matrix, giving rise to the \emph{affine Lie algebra} $\widehat{\fg}$ generated by
$\{e_i,f_i,h_i\}_{i\in \wI}$ with the defining relations~\eqref{eqn:rel 1 finite lie}--\eqref{eqn:rel 3 finite lie}.
We note that~\eqref{eq:kernel-parity} implies~that
\begin{equation*}
  c = h_0 + \sum_{i \in I} \theta_i h_i \qquad \mathrm{is\ a\  central\ element\ of} \ \ \widehat{\fg} \,.
\end{equation*}
The associated affine root system $\wDelta=\wDelta^+ \sqcup \wDelta^-$ has the following description:
\begin{align*}
  & \wDelta^+ = \big\{ \Delta^+ \times \BZ_{\geq 0} \big\}
    \sqcup \big\{ 0 \times \BZ_{>0} \big\}
    \sqcup \big\{ \Delta^- \times \BZ_{>0} \big\} ,
  \\
  & \wDelta^- = \big\{ \Delta^- \times \BZ_{\leq 0} \big\}
    \sqcup \big\{ 0 \times \BZ_{<0} \big\}
    \sqcup \big\{ \Delta^+ \times \BZ_{<0} \big\} .
\end{align*}
With this notation, we have the following root space decomposition, cf.~\eqref{eq:root.decomp}:
\begin{equation*}
  \widehat{\fg}=\widehat{\fh} \oplus \bigoplus_{\alpha \in \wDelta} \widehat{\fg}_{\alpha}
  \quad \mathrm{where} \quad \widehat{\fh}\subset \widehat{\fg} -\mathrm{Cartan\ subalgebra} \,.
\end{equation*}
The rich theory of affine Lie algebras is mainly based on the following key result:

\begin{claim}\label{claim:affine-vs-loop-Lie}
There exists a Lie algebra isomorphism:
\begin{equation*}
  \widehat{\fg}/(c)\ \iso\ L\fg
\end{equation*}
determined on the generators by the following formulas:
\begin{align*}
  & e_i \mapsto e_i \otimes t^0 & &
    f_i \mapsto f_i \otimes t^0 & &
    h_i \mapsto h_i \otimes t^0 \qquad \forall\, i\in I \,,\\
  & e_0 \mapsto f_\theta \otimes t^1 & &
    f_0 \mapsto e_\theta \otimes t^{-1} & &
    h_0 \mapsto -[e_\theta,f_\theta]\otimes t^0 \,,
\end{align*}
where $e_\theta$ and $f_\theta$ are root vectors of degrees $\theta$ and $-\theta$, respectively.
\end{claim}


\subsection{Affine Weyl groups}
\

We have already mentioned in Remark~\ref{rem:Papi's bijection} that convex orders of $\Delta^+$ are
in $1$-to-$1$ correspondence with reduced decompositions of the longest element of the finite Weyl
group $W$ associated to $\fg$. To define the latter, consider the \emph{coroot lattice}:
\begin{equation*}
  Q^\vee =\, \bigoplus_{i\in I} \BZ\cdot \bs_i^\vee
\end{equation*}
where for any $\alpha \in \Delta^+$ the corresponding \emph{coroot} $\alpha^\vee$ is defined via
$\alpha^\vee = \frac {2\alpha}{(\alpha, \alpha)}$. The finite Weyl group $W$, i.e.\ the abstract
Coxeter group associated to the Cartan matrix $(a_{ij})_{i,j\in I}$, acts faithfully on the coroot
lattice $Q^\vee$ and the root lattice $Q$:
\begin{equation}\label{eqn:action coroots}
  W \curvearrowright Q^\vee \qquad \text{and} \qquad W \curvearrowright Q
\end{equation}
via the following assignments ($\forall\, i \in I,\, \mu\in Q^\vee,\, \lambda\in Q$):
\begin{equation*}
  s_i(\mu) = \mu - (\bs_i, \mu) \bs_i^\vee\ \qquad \text{and} \qquad
  s_i(\lambda) = \lambda - (\lambda,\bs^\vee_i) \bs_i \,.
\end{equation*}

In the present setup, we need the \emph{affine Weyl group}, defined as the semidirect product
$\wW = W \ltimes Q^\vee$ with respect to the action \eqref{eqn:action coroots}. It is well-known
that $\wW$ is also the Coxeter group associated to the Cartan matrix $(a_{ij})_{i,j\in \wI}$.
In other words, the affine Weyl group is generated by the symbols $\{s_i\}_{i\in \wI}$ defined by:
\begin{equation*}
  s_0=(s_\theta,-\theta^\vee) \qquad \mathrm{and} \qquad s_i=(s_i,0) \quad \forall\, i \in I \,.
\end{equation*}
The affine analogue of the action $W \curvearrowright Q$ from~\eqref{eqn:action coroots} is
\begin{equation}\label{eqn:action affine coroots}
  \wW \curvearrowright \wQ \,,
\end{equation}
where the generators of the affine Weyl group act by the following formulas:
\begin{align*}
  & s_i(\lambda,d) = \left(\lambda - (\lambda,\bs^\vee_i) \bs_i , d \right) \quad \forall\, i \in I \,,
    \\
  & s_0(\lambda,d) = \left(\lambda - (\lambda,\theta^\vee) \theta , d + (\lambda,\theta^\vee) \right)
\end{align*}
for all $(\lambda,d) \in Q \times \BZ \simeq \wQ$, see~\eqref{eqn:finite vs affine roots}.
An important feature of the affine Weyl group is that it contains a large commutative subalgebra
$1 \ltimes Q^\vee \subset \wW$ which acts on the affine root lattice $\wQ\simeq Q\times \BZ$ by translations:
\begin{equation}\label{eqn:coroot action}
  \wmu(\lambda,d) = \left(\lambda, d-(\lambda, \mu) \right)
  \qquad \forall\, \mu\in Q^\vee,\, \lambda\in Q,\, d\in \BZ \,.
\end{equation}
Henceforth, we write $\wmu$ for the element $1 \ltimes \mu \in \wW$ and call it a \emph{translation element}.

Finally, we also need to consider the \emph{extended affine Weyl group}, defined as the
semidirect product $\weW = W \ltimes P^\vee$, where $P^\vee$ is the \emph{coweight lattice}.
Thus $P^\vee=\, \bigoplus_{i\in I} \BZ\cdot \omega_i^\vee$ with the fundamental coweights
$\omega^\vee_i$ dual to the simple roots:
\begin{equation}\label{eq:coweight}
  (\alpha_j,\omega^\vee_i)=\delta_{ij} \,.
\end{equation}
In particular, $Q^\vee$ is a finite index subgroup of $P^\vee$. It is well-known that:
\begin{equation}\label{eqt:extended vs usual}
  \weW\simeq \CT\ltimes \wW
\end{equation}
where the finite subgroup $\CT$ of $\weW$ is naturally identified with a subgroup of automorphisms
of the Dynkin diagram of $\hg$. The semi-direct product~\eqref{eqt:extended vs usual} is such that
$\tau s_i=s_{\tau(i)}\tau$ for any $\tau\in \CT$ and $i\in \wI$.
Finally, the action~\eqref{eqn:action affine coroots} extends to:
\begin{equation*}
  \weW \curvearrowright \wQ
\end{equation*}
via $\tau(\alpha_i)=\alpha_{\tau(i)}$ for $\tau\in \CT$, $i\in \wI$.
We still have the following formula, akin~to~\eqref{eqn:coroot action}:
\begin{equation}\label{eqn:coweight action}
  \wmu(\lambda,d) = \left(\lambda, d-(\lambda, \mu) \right)
  \qquad \forall\, \mu\in P^\vee,\, \lambda\in Q,\, d\in \BZ \,.
\end{equation}


\subsection{Reduced decompositions}\label{sub:convex affine order}
\

Recall that the \underline{length} of an element $x \in \wW$, denoted by $l(x)\in \BN$,
is the smallest number $l \in \BN$ such that we can write $x = s_{i_{1-l}} \dots s_{i_0}$ for various
$i_{1-l}, \dots, i_0 \in \wI$. Every such factorization is called a \emph{reduced decomposition} of $x$.
Given such a reduced decomposition, the \emph{terminal subset} of the affine root system~is:
\begin{equation}\label{eqn:terminal set}
  E_x=\Big\{ s_{i_0}s_{i_{-1}}\dots s_{i_{k+1}}(\bs_{i_{k}}) \,\Big|\, 0\geq k > -l \Big\} \subset \widehat{\Delta} \,.
\end{equation}
It is well-known that $E_x$ is independent of the reduced decomposition of $x$, and consists of
the positive affine roots (all with multiplicity one) that are mapped to negative ones under the
action of $x$:
\begin{equation}\label{eqn:terminal description}
  E_x = \left\{ \widetilde{\lambda}\in \widehat{\Delta}^+ \,\Big|\, x(\widetilde{\lambda})\in \widehat{\Delta}^- \right\} .
\end{equation}
In particular, we get the following description of the length of $x$:
\begin{equation*}
  l(x) = \# \left\{ \widetilde{\lambda}\in \widehat{\Delta}^+ \,\Big|\, x(\widetilde{\lambda})\in \widehat{\Delta}^- \right\} .
\end{equation*}
The aforementioned length function $l\colon \wW\to \BN$ naturally extends to $\weW$ via
  $$ l(\tau w)=l(w) \quad \mathrm{for\ any} \quad \tau\in \CT \,,\, w\in \wW \,. $$
Thus, the length $l(x)$ of $x\in \weW$ is the smallest number $l$ such that we can write:
\begin{equation}\label{eqn:extended reduced decomposition}
  x = \tau s_{i_{1-l}} \dots s_{i_0}
\end{equation}
for various $i_{1-l},\dots,i_0\in \wI$ and (a uniquely determined) $\tau\in \CT$. Given a reduced decomposition of
$x\in \weW$ as in~\eqref{eqn:extended reduced decomposition} with $l=l(x)$, define $E_x$ via~\eqref{eqn:terminal set}.
We note that $E_x$ is still described via~\eqref{eqn:terminal description} since $\tau$ acts by permuting negative
affine roots. Therefore, $E_x$ is independent of the reduced decomposition of $x$ and we still have:
\begin{equation*}
  l(x) = \# \left\{ \widetilde{\lambda}\in \widehat{\Delta}^+ \,\Big|\, x(\widetilde{\lambda})\in \widehat{\Delta}^- \right\} .
\end{equation*}

The following result is well-known (cf.~\cite[Proposition 3.9]{NT}):

\begin{proposition}\label{prop:length for lattice}
For any $\mu \in P^\vee$ such that $(\alpha_i,\mu) \in \BZ_{>0}$ for all $i \in I$, we have
\begin{equation}\label{eq:terminal-tranlsation}
  E_{\wmu} = \Big\{(\alpha,d) \,\Big|\, \alpha\in \Delta^+, 0\leq d<(\alpha,\mu) \Big\} ,
\end{equation}
and consequently
  $$ l(\wmu) = \sum_{\alpha\in \Delta^+} (\alpha, \mu) \,. $$
\end{proposition}


\subsection{Identification of orders}\label{sub:Lyndon via affWeyl}
\

We start by recalling the classical construction of~\cite{B2}. Pick any $\mu \in P^\vee$ such that
$(\alpha_i, \mu) \in \BZ_{>0}$ for all $i \in I$. Let $l=l(\wmu)$ and consider any reduced decomposition:
\begin{equation}\label{eqn:reduced-rho}
  \wmu = \tau s_{i_{1-l}}s_{i_{2-l}}\dots s_{i_0} \,.
\end{equation}
Extend $i_{1-l},\dots,i_0$ to a $\tau$-\emph{quasiperiodic} bi-infinite sequence $\{i_k\}_{k \in \BZ}$ via:
\begin{equation}\label{eqn:tau-twisted sequence}
  i_{k+l}=\tau(i_k) \qquad \forall\, k\in \BZ \,.
\end{equation}
To such a bi-infinite sequence~(\ref{eqn:tau-twisted sequence}), one assigns the following
bi-infinite sequence of affine roots:
\begin{equation}\label{eqn:beta-roots}
  \beta_k =
  \begin{cases}
    s_{i_1} s_{i_2} \dots s_{i_{k-1}}(-\alpha_{i_k}) &\text{if } k > 0 \\
    s_{i_0} s_{i_{-1}} \dots s_{i_{k+1}}(\alpha_{i_k}) &\text{if } k \leq 0
  \end{cases} \,.
\end{equation}
According to~\cite{P1,P2}, the sequences:
\begin{align}
  & \beta_1 > \beta_2 > \beta_3 > \cdots \label{eqn:beta positive} \\
  & \beta_0 < \beta_{-1} < \beta_{-2} < \cdots \label{eqn:beta negative}
\end{align}
give convex orders of the sets $\Delta^+ \times \BZ_{<0}$ and $\Delta^+ \times \BZ_{\geq 0}$, respectively.
We note that $\beta_{k+l} = \wmu(\beta_k)$ for any $k\in \BZ$. Due to~(\ref{eqn:coweight action}),
if $\beta_k = (\alpha,d)$ and $\beta_{k+l} = (\alpha',d')$, then
\begin{equation}\label{eqn:beta periodicity}
  \alpha = \alpha' \quad \text{ and } \quad d = d'+(\alpha,\mu) \,.
\end{equation}
This reveals a periodicity of the entire set $\Delta^+ \times \BZ$, not
just $\Delta^+ \times \BZ_{<0}$ and $\Delta^+ \times \BZ_{\geq 0}$.

Evoking the setup of Section~\ref{sec:loop algebra}, let us consider
\begin{equation}\label{eq:specific-mu}
  \mu = \, \sum_{i\in I} c_i\omega^\vee_i
\end{equation}
so that $f(\alpha)=(\alpha,\mu)$ for any $\alpha\in \Delta^+$, cf.~\eqref{eqn:weighted height} and~\eqref{eq:coweight}.
The following is the first main result of this Section, which naturally generalizes~\cite[Theorem 3.14]{NT}:

\begin{theorem}\label{thm:weyl to lyndon}
There exists a reduced decomposition of $\wmu\in \weW$ such that:

\medskip

\begin{itemize}[leftmargin=*]

\item
the order~\eqref{eqn:beta positive} of the roots $\{(\alpha,d) \,|\, \alpha\in \Delta^+, d<0\}$ matches
the lexicographic order of the standard Lyndon loop words $\ell(\alpha,-d)$ via~(\ref{eqn:induces affine}),

\medskip
\item
the order~\eqref{eqn:beta negative} of the roots $\{(\alpha,d) \,|\, \alpha\in \Delta^+,d\geq 0\}$ matches
the lexicographic order of the standard Lyndon loop words $\ell(\alpha,-d)$ via~(\ref{eqn:induces affine}).
\end{itemize}
\end{theorem}

\begin{proof}
Recall the finite subset $L=\{(\alpha,d) \,|\, \alpha\in \Delta^+, 0\leq d<f(\alpha)\}\subset \widehat{\Delta}^+$
from \eqref{eqn:L-chunk}, ordered via:
\begin{equation}\label{eqn:order-L}
  (\alpha,d)<(\beta,e) \quad \Longleftrightarrow \quad \ell(\alpha,-d)<\ell(\beta,-e) \,.
\end{equation}
If $(\alpha,d),(\beta,e)\in L$ with $(\alpha,d)<(\beta,e)$ and $(\alpha+\beta,d+e)\in \widehat{\Delta}$,
then clearly $(\alpha+\beta,d+e)\in L$, as well as $(\alpha,d) < (\alpha+\beta,d+e) < (\beta,e)$,
due to Proposition~\ref{prop:convex loop}.

Furthermore, we claim that if $\widetilde{\lambda},\widetilde{\mu}\in \widehat{\Delta}^+$ with
$\widetilde{\lambda} + \widetilde{\mu} \in L$, then at least one of $\widetilde{\lambda}$ or $\widetilde{\mu}$
belongs to $L$ and is $<\widetilde{\lambda} + \widetilde{\mu}$. There are two cases to consider:
\begin{enumerate}

\item
If $\widetilde{\lambda}=(\alpha,d)$, $\widetilde{\mu}=(\beta,e)$ with $\alpha,\beta\in \Delta^+$ and $d,e\geq 0$,
we can assume without loss of generality that $\ell(\alpha,-d)<\ell(\beta,-e)$. By Proposition~\ref{prop:convex loop},
we have $\ell(\alpha,-d)<\ell(\alpha+\beta,-d-e)<\ell(\beta,-e)$.
It remains to prove $d<f(\alpha)$. If not, then $e<f(\beta)$ as $d+e<f(\alpha+\beta)$.
Hence, the first letter of $\ell(\alpha,-d)$ has a relative exponent $\leq -1$ and the first
letter of $\ell(\beta,-e)$ has a relative exponent $>-1$, due to Corollary~\ref{cor:first letter}
and Proposition~\ref{prop:peridocitiy}, which contradicts $\ell(\alpha,-d)<\ell(\beta,-e)$.

\item
In the remaining case, we may assume $\widetilde{\lambda}=(\alpha+\beta,d), \widetilde{\mu}=(-\beta,e)$, so that
$\alpha,\beta,\alpha+\beta \in \Delta^+$ and $d\geq 0, e>0$. Then $d<d+e<f(\alpha)<f(\alpha+\beta)$, so that
$\widetilde{\lambda}\in L$. It remains to verify $\ell(\alpha+\beta,-d)<\ell(\alpha,-d-e)$.
Since $(\alpha+\beta,-d)=(\beta,e)+(\alpha,-d-e)$, it suffices to prove
$\ell(\beta,e)<\ell(\alpha,-d-e)$, due to Proposition~\ref{prop:convex loop}.
But applying Corollary~\ref{cor:first letter} once again, we see that the exponent of the first letter in
$\ell(\beta,e)$ is $>0$, while the exponent of the first letter in $\ell(\alpha,-d-e)$ is $\leq 0$, hence,
indeed $\ell(\beta,e)<\ell(\alpha,-d-e)$.

\end{enumerate}

Invoking~\cite{P} (which also applies to finite subsets in affine root systems), we get:

\begin{enumerate}

\item[(I)]
there is a unique element $x\in \wW$ such that $L=E_x$,

\item[(II)]
the order of $L$ arises via a unique reduced decomposition of $x$, where the set $E_x$
of~\eqref{eqn:terminal set} is ordered via
  $\alpha_{i_0} < s_{i_0}(\alpha_{i_{-1}}) < \dots < s_{i_0}s_{i_{-1}}\dots s_{i_{2-l}}(\alpha_{i_{1-l}})$.

\end{enumerate}
However, as follows from~\eqref{eq:terminal-tranlsation}, we have
\begin{equation}\label{eq:L-terminal}
  L = E_{\wmu} = \Big\{\beta_0,\beta_{-1},\dots,\beta_{1-l}\Big\} .
\end{equation}
There is a unique $\tau\in \CT$ such that $\tau^{-1}\wmu\in \wW$. Thus, we obtain $L = E_{\wmu} = E_{\tau^{-1}\wmu}$.
Therefore, in view of the uniqueness statement of~(I), the result of~(II) implies that there exists a reduced
decomposition~(\ref{eqn:reduced-rho}) of $\wmu$ such that the ordered finite sequence
$\beta_0<\beta_{-1}<\dots<\beta_{1-l}$ exactly coincides with $L$ ordered via~(\ref{eqn:order-L}).

The proof of Theorem~\ref{thm:weyl to lyndon} now follows by combining~(\ref{eqn:beta periodicity}),
Proposition~\ref{prop:peridocitiy}, and Theorem~\ref{thm:ExpRule}, precisely as in~\cite{NT}. Indeed,
let us split $\Delta^+ \times \BZ$ into the blocks:
  $$ L_N = \Big\{(\alpha,d) \,\Big|\, \alpha\in \Delta^+,\, N\cdot f(\alpha) \leq d < (N+1)f(\alpha) \Big\} $$
so that
\begin{equation*}
  \bigsqcup_{N\geq 0} L_N = \Delta^+\times \BZ_{\geq 0} = \{\beta_k\}_{k\leq 0} \,, \qquad
  \bigsqcup_{N<0} L_N = \Delta^+\times \BZ_{<0} = \{\beta_k\}_{k>0} \,.
\end{equation*}
According to~(\ref{eqn:beta periodicity}) and $L_0=L=\{\beta_0,\dots,\beta_{1-l}\}$, we have:
  $$ L_N = \Big\{\beta_{-Nl},\beta_{-Nl-1},\dots,\beta_{1-(N+1)l} \Big\} \qquad \forall\, N\in \BZ \,. $$
For any $(\alpha,d)\in L_N$, the relative exponent of the first letter in $\ell(\alpha,-d)$ lies
in $(-N-1;-N]$, due to Corollary~\ref{cor:first letter} and Proposition~\ref{prop:peridocitiy}. Thus,
for any $(\alpha,d)\in L_M$, $(\beta,e)\in L_N$ with $M>N$, we have $\ell(\alpha,-d)>\ell(\beta,-e)$.
As for the affine roots from the same block, consider $\beta_{r-Nl},\beta_{s-Nl}\in L_N$ with $1-l\leq s<r\leq 0$.
If $\beta_r=(\alpha,d)$ and $\beta_s=(\beta,e)$, then $\beta_{r-Nl}=(\alpha, d+N\cdot f(\alpha))$ and
$\beta_{s-Nl}=(\beta,e+N\cdot f(\beta))$, due to~(\ref{eqn:beta periodicity}). On the other hand,
the words $\ell(\alpha,-d-N\cdot f(\alpha))$ and $\ell(\beta,-e-N\cdot f(\beta))$ are obtained from
$\ell(\alpha,-d)$ and $\ell(\beta,-e)$, respectively, by decreasing each relative exponent by $N$,
due to Proposition~\ref{prop:peridocitiy}. Since the latter operation obviously preserves the
lexicographic order, and $\ell(\alpha,-d)<\ell(\beta,-e)$ as a consequence of $r > s$, we obtain
the required inequality $\ell(\alpha,-d-N\cdot f(\alpha))<\ell(\beta,-e-N\cdot f(\beta))$.
\end{proof}

\begin{remark}\label{rem:halves-into-total}
Since $\ell(\alpha,-d) < \ell(\beta,-e)$ if $d < 0 \leq e$, a consequence of Corollary~\ref{cor:first letter},
we actually have the stronger result that the order of $\Delta^+ \times \BZ$ given by:
\begin{equation*}
  \dots < \beta_3 < \beta_{2} < \beta_1 < \beta_0 < \beta_{-1} < \beta_{-2} < \cdots
\end{equation*}
matches the lexicographic order of the standard Lyndon loop words $\ell(\alpha,-d)$.
\end{remark}

In the next Section, we shall need a certain generalization of~\eqref{eq:L-terminal}. To this end,
for any $i\in I$ and $d\geq 0$, we define the subset $L_{<(i,d)}$ of $\Delta^+\times \BZ$ via
\begin{equation*}
  L_{<(i,d)} =
  \Big\{ (\alpha,p) \,\Big|\, \alpha\in \Delta^+,\, p\in \BZ_{\geq 0},\, \ell(\alpha,-p)<\ell(\alpha_i,-d) \Big\} .
\end{equation*}
We also define a collection of nonnegative integers $\{p_j\}_{j\in I}$ via:
\begin{equation}\label{eq:p-constants}
  p_j=
  \begin{cases}
    d &\text{if } j=i \\
    \left\lceil \frac{d\cdot c_j}{c_i} \right\rceil &\text{if } \frac{d\cdot c_j}{c_i}\notin \BZ \\
    \frac{d\cdot c_j}{c_i} &\text{if } \frac{d\cdot c_j}{c_i}\in \BZ \ \ \mathrm{and}\ \ j>i \\
    \frac{d\cdot c_j}{c_i} + 1 &\text{if } \frac{d\cdot c_j}{c_i}\in \BZ \ \ \mathrm{and}\ \ j<i
  \end{cases} \,.
\end{equation}
Finally, for any positive root $\alpha=\sum_{i\in I} k_i\alpha_i\in \Delta^+$, we define $p(\alpha)\in \BN$ via
  $$ p(\alpha)=\sum_{i\in I} k_i p_i \,. $$

\begin{proposition}\label{prop:terminal-segment}
For any $i\in I$ and $d\geq 0$, we have
\begin{equation*}
  L_{<(i,d)}=\Big\{ (\alpha,p) \,\Big|\, \alpha\in \Delta^+, 0\leq p<p(\alpha) \Big\} .
\end{equation*}
\end{proposition}

\begin{proof}
First, let us prove that $\ell(\alpha,-p)<\ell(\alpha_i,-d)=i^{(-d)}$ implies $p<p(\alpha)$.
Let $j^{(-e)}$ be the first letter of $\ell(\alpha,-p)$, so that $j^{(-e)}<i^{(-d)}$. Hence,
$e/c_j\leq d/c_i$ and the inequality is strict if $j\geq i$. This is equivalent to $e<p_j$,
due to the definition~\eqref{eq:p-constants}.
Then for any letter $\imath^{(-s)}\in \ell(\alpha,-p)$, we have $\imath^{(-s+1)}\leq j^{(-e)}< i^{(-d)}$
with the first inequality due to Theorem~\ref{thm:ExpRule}. As above, this implies $s-1< p_\imath$, so that
$s\leq p_\imath$. Summing all these inequalities, we obtain the desired inequality $p<p(\alpha)$.

Let us prove the opposite implication by contradiction: assume that $\ell(\alpha,-p)>\ell(\alpha_i,-d)$
for some $\alpha\in \Delta^+$ and $p<p(\alpha)$. Let $j^{(-e)}$ be the first letter of $\ell(\alpha,-p)$,
so that $j^{(-e)}\geq i^{(-d)}$.
As $\ell(\alpha,-p)$ is Lyndon, any letter $\imath^{(-s)}\in \ell(\alpha,-p)$ satisfies $j^{(-e)}\leq \imath^{(-s)}$.
Therefore, $s/c_\imath\geq d/c_i$ and the inequality is strict for $\imath<i$. Thus, $s\geq p_\imath$. Summing
all these inequalities, we obtain $p\geq p(\alpha)$, a contradiction.
This completes our proof of $\ell(\alpha,-p)<\ell(\alpha_i,-d)$ for any $0\leq p<p(\alpha)$.
\end{proof}

In view of Proposition~\ref{prop:length for lattice}, the above result can be recast as follows:

\begin{proposition}\label{prop:terminal-segment-2}
For any $i\in I$ and $d\geq 0$, we have $L_{<(i,d)}=E_{\widehat{\omega_{i,d}}}$, where
\begin{equation}\label{eq:omega-pi}
  \omega_{i,d}=\sum_{j\in I} p_j\omega^\vee_j \in P^\vee
\end{equation}
with $p_j$'s defined in~\eqref{eq:p-constants}.
\end{proposition}


\medskip

\section{Quantum groups and PBW bases}\label{sec:quantum}

In this Section, we combine the results of Subsection~\ref{sub:Lyndon via affWeyl} with the PBW-type bases \cite{B,B2}
of quantum affine algebras (in the Drinfeld-Jimbo realization) to produce a family of PBW-type combinatorial bases of
quantum loop algebras (in the new Drinfeld realization), thus generalizing the construction of~\cite{L} for the finite type.


\subsection{Quantum groups}\label{sub:Drinfeld-Jimbo}
\

We shall follow the notation of Subsection~\ref{sub:LR-bijection}, corresponding to a simple finite-dimensional $\fg$.
Consider the $q$-numbers, $q$-factorials, and $q$-binomial coefficients:
\begin{equation*}
  [k]_i = \frac {q_i^k - q_i^{-k}}{q_i - q_i^{-1}} \,, \qquad
  [k]!_i = [1]_i \dots [k]_i \,, \qquad
  {n \choose k}_i = \frac {[n]!_i}{[k]!_i [n-k]!_i}
\end{equation*}
for any $i \in I$, where $q_i = q^{\frac {d_{ii}}2}$.

\begin{definition}\label{def:finite quantum group}
The Drinfeld-Jimbo quantum group of $\fg$, denoted by $\uu$, is an associative $\BQ(q)$-algebra generated by
$\{e_i, f_i, \ph_i^{\pm 1}\}_{i \in I}$ subject to  the following defining relations (for all $i,j\in I$):
\begin{equation}\label{eqn:rel 1 finite}
  \sum_{k=0}^{1-a_{ij}} (-1)^k {1-a_{ij} \choose k}_i e_i^k e_j e_i^{1-a_{ij}-k} = 0  \qquad \mathrm{if }\ i\ne j \,,
\end{equation}
\begin{equation}\label{eqn:rel 2 finite}
  \ph_i e_j = q^{d_{ij}} e_j \ph_i \,, \qquad \ph_i\ph_j = \ph_j \ph_i \,,
\end{equation}
as well as the opposite relations with $e$'s replaced by $f$'s, and finally the relation:
\begin{equation}\label{eqn:rel 3 finite}
  [e_i, f_j] = \delta_{ij} \cdot \frac {\ph_i - \ph_i^{-1}}{q_i - q_i^{-1}} \,.
\end{equation}
\end{definition}

The algebra $\uu$ is naturally $Q$-graded via
  $$ \deg e_i = \alpha_i \,, \qquad \deg \ph_i = 0 \,, \qquad \deg f_i = -\alpha_i \,. $$
Furthermore, it admits the triangular decomposition~\eqref{eqn:decomp intro 2}:
\begin{equation*}
  \uu = \uup \otimes \uuo \otimes \uum \,,
\end{equation*}
where $\uup$, $\uuo$, and $\uum$ are the subalgebras of $\uu$ generated by the $e_i$'s, $\ph^{\pm 1}_i$'s,
and $f_i$'s, respectively. In fact, the associative algebra $\uup$ is generated by $e_i$'s with the defining
relations~(\ref{eqn:rel 1 finite}), cf.\ e.g.~\cite[\S4.21]{J}.

If we write $\ph_i = q_i^{h_i}$ and take the limit $q \rightarrow 1$, then $\uu$ degenerates to $U(\fg)$.
It is thus natural that many features of the latter also admit $q$-deformations. For example, let us recall
the notion of standard Lyndon words from Subsections~\ref{sub:finite words}--\ref{sub:standard words}, and
consider the following $q$-version of Definition~\ref{def:bracketing lyndon} and the construction~\eqref{eqn:bracket.word}:

\begin{definition}\label{def:quantum bracketing}(\cite{L})
For any word $w$, define $e_w \in U_q(\fn^+)$ by:
  $$ e_{[i]} = e_i $$
for all $i \in I$, and then recursively by:
\begin{equation*}
  e_{\ell} = \left[ e_{\ell_1}, e_{\ell_2} \right]_q =
  e_{\ell_1} e_{\ell_2} - q^{(\deg \ell_1, \deg \ell_2)} e_{\ell_2} e_{\ell_1}
\end{equation*}
if $\ell$ is a Lyndon word with the costandard factorization~\eqref{eqn:costandard factorization}, and:
\begin{equation*}
  e_w = e_{\ell_1} \dots e_{\ell_k}
\end{equation*}
if $w$ is an arbitrary word with the canonical factorization $\ell_1 \dots \ell_k$,
as in~\eqref{eqn:canonical factorization}.
\end{definition}

We also define $f_w\in U_q(\fn^-)$ by replacing $e$'s by $f$'s in the above Definition.
Then we have the following natural $q$-deformation of the PBW theorem \eqref{eqn:pbw lie}:

\begin{theorem}\label{thm:PBW quantum finite}
We have:
\begin{equation*}
\begin{split}
  & U_q(\fn^+) \ =
  \bigoplus^{k \in \BN}_{\ell_1 \geq \dots \geq \ell_k \text{ standard Lyndon words}}
  \BQ(q) \cdot e_{\ell_1} \dots e_{\ell_k} \ = \\
  & \qquad \qquad \qquad \qquad \qquad
  \bigoplus_{w \text{--standard words}} \BQ(q) \cdot e_w \,.
\end{split}
\end{equation*}
The analogous result also holds with $\fn^+ \leftrightarrow \fn^-$ and $e \leftrightarrow f$.
\end{theorem}

This result is a consequence of the usual PBW theorem for $U_q(\fn^\pm)$, since $e_\ell$'s are simply
renormalizations of the standard root vectors constructed in \cite{Lu} using the braid group action,
according to~\cite[Theorem 28]{L} (cf.\ also~\cite[Section 5.5]{NT}).


\subsection{Quantum loop algebras}\label{sub:qaffine}
\

To introduce a loop version of the above algebras, consider the generating series
$$
  e_i(z) = \sum_{k \in \BZ} \frac {e_{i,k}}{z^k} \,, \qquad
  f_i(z) = \sum_{k \in \BZ} \frac {f_{i,k}}{z^k} \,, \qquad
  \ph^\pm_i(z) = \sum_{l = 0}^\infty \frac {\ph^\pm_{i,l}}{z^{\pm l}}
$$
as well as the formal delta function $\delta(z) = \sum_{k \in \BZ} z^k$. For any $i,j \in I$, we set:
\begin{equation*}
  \zeta_{ij} \left(\frac zw \right) = \frac {z - wq^{-d_{ij}}}{z - w} \,.
\end{equation*}

\begin{definition}\label{def:quantum loop}
The quantum loop group (in the new Drinfeld realization) of $\fg$, denoted by $\UU$, is an
associative $\BQ(q)$-algebra generated by $\{e_{i,k}, f_{i,k}, \ph_{i,l}^\pm\}_{i \in I}^{k\in \BZ, l\in \BN}$
subject to the following defining relations (for all $i,j \in I$):
\begin{equation}\label{eqn:rel 0 affine}
  e_i(z) e_j(w) \zeta_{ji} \left(\frac wz \right) =\, e_j(w) e_i(z) \zeta_{ij} \left(\frac zw \right) ,
\end{equation}
\begin{multline}\label{eqn:rel 1 affine}
  \sum_{\sigma \in S(1-a_{ij})} \sum_{k=0}^{1-a_{ij}} (-1)^k {1-a_{ij} \choose k}_i \cdot  \\
  \qquad \qquad
  e_i(z_{\sigma(1)})\dots e_i(z_{\sigma(k)}) e_j(w) e_i(z_{\sigma(k+1)}) \dots e_i(z_{\sigma(1-a_{ij})}) = 0
  \qquad \mathrm{if}\ i\ne j \,,
\end{multline}
\begin{equation}\label{eqn:rel 2 affine}
  \ph_i^\pm(z) e_j(w) \zeta_{ji} \left(\frac wz \right) = e_j(w) \ph_i^\pm(z) \zeta_{ij} \left(\frac zw \right) ,
\end{equation}
\begin{equation}\label{eqn:rel 2 affine bis}
  \ph_{i}^{\pm}(z) \ph_{j}^{\pm'}(w) = \ph_{j}^{\pm'}(w) \ph_{i}^{\pm}(z) \,, \qquad
  \ph_{i,0}^+ \ph_{i,0}^- = 1 \,,
\end{equation}
as well as the opposite relations with $e$'s replaced by $f$'s, and finally the relation:
\begin{equation}\label{eqn:rel 3 affine}
  \left[ e_i(z), f_j(w) \right] =
  \frac {\delta_{ij} }{q_i - q_i^{-1}}\delta \left(\frac zw \right) \cdot  \Big( \ph_i^+(z) - \ph_i^-(w) \Big) .
\end{equation}
\end{definition}

The algebra $\UU$ is naturally $Q \times \BZ$-graded via
  $$ \deg e_{i,k} = (\alpha_i,k) \,, \qquad \deg \ph_{i,l}^\pm = (0,\pm l) \,, \qquad \deg f_{i,k} = (-\alpha_i,k) $$
for $i\in I, k\in \BZ, l\in \BN$. If $x\in \UU$ has a $Q\times \BZ$-degree $\deg x=(\alpha,d)$, then we~set
\begin{equation*}
  \hdeg x = \alpha \qquad \text{and} \qquad \vdeg x = d \,,
\end{equation*}
and call these the \emph{horizontal} and the \emph{vertical} degrees of $x$, respectively, cf.~\eqref{eq:hor and vert words}.

\noindent
Finally, the algebra $\UU$ also admits the triangular decomposition (cf.~\cite[\S3.3]{He}):
\begin{equation}\label{eqn:triangular loop}
  \UU = \UUp \otimes \UUo \otimes \UUm \,,
\end{equation}
where $\UUp$, $\UUo$, and $\UUm$ are the subalgebras of $\UU$ generated by the $e_{i,k}$'s, $\ph_{i,l}^\pm$'s,
and $f_{i,k}$'s, respectively. In fact, the associative algebra $\UUp$ is generated by $e_{i,k}$'s with the defining
relations~(\ref{eqn:rel 0 affine},~\ref{eqn:rel 1 affine}).

Let us now present a loop version of Definition~\ref{def:quantum bracketing}:

\begin{definition}\label{def:loop quantum bracketing}
For any loop word $w$, define $e_w \in \UUp$ and $f_w \in \UUm$ by:
  $$ e_{[i^{(d)}]} = e_{i,d} \qquad \text{and} \qquad f_{[i^{(d)}]} = f_{i,-d} $$
for all $i \in I$, $d \in \BZ$, and then recursively by:
\begin{align}\label{eqn:quantum bracketing lyndon affine}
  & e_{\ell} = \left[ e_{\ell_1}, e_{\ell_2} \right]_q =
    e_{\ell_1} e_{\ell_2} - q^{(\mathrm{hdeg}\, \ell_1, \mathrm{hdeg}\, \ell_2)} e_{\ell_2} e_{\ell_1} \,, \\
  & f_{\ell} = \left[ f_{\ell_1}, f_{\ell_2} \right]_q =
    f_{\ell_1} f_{\ell_2} - q^{(\mathrm{hdeg}\, \ell_1, \mathrm{hdeg}\, \ell_2)} f_{\ell_2} f_{\ell_1}
\end{align}
if $\ell$ is a Lyndon loop word with the costandard factorization~\eqref{eqn:costandard factorization}, and:
\begin{equation*}
  e_w = e_{\ell_1} \dots e_{\ell_k} \qquad \text{and} \qquad f_w = f_{\ell_1} \dots f_{\ell_k}
\end{equation*}
if $w$ is an arbitrary loop word with the canonical factorization $\ell_1 \dots \ell_k$, as in~\eqref{eqn:canonical factorization}.
\end{definition}

Note that $\deg e_w = -\deg f_w = \deg w$ for all loop words $w$. The following is the main result of
this Section, which generalizes~\eqref{eqn:pbw lie loop} as well as Theorem~\ref{thm:PBW quantum finite}:

\begin{theorem}\label{thm:PBW quantum loop}
We have:
\begin{multline*}
  \UUp \ =
  \bigoplus^{k \in \BN}_{\ell_1 \geq \dots \geq \ell_k \text{ standard Lyndon loop words}}
    \BQ(q) \cdot e_{\ell_1} \dots e_{\ell_k} \ = \\
  \bigoplus_{w \text{--standard loop words}} \BQ(q) \cdot e_w \,.
\end{multline*}
The analogous result also holds with $\UUp \leftrightarrow \UUm$ and $e \leftrightarrow f$.
\end{theorem}

The proof of this result occupies the rest of this Section. While it looks similar to the proof
of~\cite[Theorem~4.24]{NT}, we shall crucially utilize Proposition~\ref{prop:terminal-segment-2}.


\subsection{Quantum affine algebras}\label{sub:finite group}
\

Let us recall the notion of Drinfeld-Jimbo quantum affine algebras and their relation to quantum
loop algebras $\UU$. We use the notations of Subsection~\ref{sub:affine Weyl}.

\begin{definition}\label{def:quantum affine}
The Drinfeld-Jimbo quantum affine algebra of $\widehat{\fg}$, denoted by $\VV$, is defined
exactly as $\uu$ in Definition \ref{def:finite quantum group}, but using $\wI$ instead of $I$.
\end{definition}

Let $\VVp, \VVo,\VVm$ be the subalgebras generated by the $e_i$'s, $\ph_i^{\pm 1}$'s, $f_i$'s, respectively
(with $i \in \wI$). We have a triangular decomposition analogous to~\eqref{eqn:decomp intro 2}:
\begin{equation}\label{eqn:triangular affine}
  \VV = \VVp \otimes \VVo \otimes \VVm \,.
\end{equation}
The algebra $\VV$ is naturally $\wQ\simeq Q\times \BZ$-graded via
\begin{align*}
  & \deg e_0 = \alpha_0 = (-\theta, 1) \,, & &
    \deg f_0 = -\alpha_0 = (\theta, -1) \,, & &
    \deg \ph_0 = 0 =(0,0) \,, \\
  & \deg e_i = \alpha_i = (\alpha_i,0) \,, & &
    \deg f_i = -\alpha_i = (-\alpha_i,0) \,, & &
    \deg \ph_i = 0 = (0,0)
\end{align*}
for $i\in I$, where $\theta$ is the highest root of $\Delta^+$.
Invoking the positive integers $\{\theta_i\}_{i\in I}$ introduced in~(\ref{eqn:labels}),
we note that the following element is central in $\VV$:
\begin{equation}\label{eqn:central c}
  C = \ph_0 \prod_{i \in I} \ph_i^{\theta_i} \,.
\end{equation}

Let us now recall the construction of the root vectors of $\VV$, presented in~\cite{B,Lu}. Following
Subsection~\ref{sub:Lyndon via affWeyl}, pick the coweight $\mu=\sum_{i\in I}c_i\omega^\vee_i \in P^\vee$
as in~\eqref{eq:specific-mu}, and set $\wmu=1\ltimes \mu \in \weW$. We consider the reduced decomposition:
  $$ \widehat{\mu} = \tau s_{i_{1-l}}s_{i_{2-l}} \dots s_{i_0} $$
from Theorem~\ref{thm:weyl to lyndon} with $\tau\in \CT$. Following~\eqref{eqn:tau-twisted sequence},
let us extend $\{i_k|-l<k\leq 0\}$ to a $\tau$-\emph{quasiperiodic} bi-infinite sequence $\{i_k\}_{k\in \BZ}$ via
$i_{k+l}=\tau(i_k)$ for any $k \in \BZ$. We construct the following set of positive affine roots:
\begin{equation}\label{eqn:affine real roots}
  \tbeta_k =
  \begin{cases}
    s_{i_1} s_{i_2} \dots s_{i_{k-1}}(\alpha_{i_k}) &\text{if } k > 0 \\
    s_{i_0} s_{i_{-1}} \dots s_{i_{k+1}}(\alpha_{i_k}) &\text{if } k \leq 0
  \end{cases}  \quad = \
  \begin{cases}
    -\beta_k  &\text{if } k > 0 \\
    \beta_k &\text{if } k \leq 0
  \end{cases} \,,
\end{equation}
with $\beta_k$ defined in~\eqref{eqn:beta-roots}. Following~\cite{B}, we shall order those roots as follows:
\begin{equation}\label{eqn:Becks order}
  \tbeta_0<\tbeta_{-1}<\tbeta_{-2}<\tbeta_{-3}<\dots<\tbeta_4<\tbeta_3<\tbeta_2<\tbeta_1 \,.
\end{equation}

\begin{remark}\label{rem:tbeta-convexity}
Formula \eqref{eqn:affine real roots} provides all real positive roots of $\widehat{\Delta}^+$:
\begin{equation}\label{eqn:real positive}
  \widehat{\Delta}^{\mathrm{re},+} =
  \Big\{ \Delta^+ \times \BZ_{\geq 0} \Big\} \sqcup \Big\{ \Delta^- \times \BZ_{>0} \Big\} \subset \widehat{\Delta}^+ \,.
\end{equation}
Furthermore,~\eqref{eqn:Becks order} induces convex orders on the corresponding halves:
\begin{equation}\label{eq:halves}
  \Delta^+ \times \BZ_{\geq 0}=\Big\{ \tbeta_0 < \tbeta_{-1} <\tbeta_{-2}< \cdots \Big\}
  \,,
  \Delta^- \times \BZ_{> 0}=\Big\{ \dots< \tbeta_3 < \tbeta_2 < \tbeta_1 \Big\}.
\end{equation}
To have a complete theory, in particular for the PBW theorem of~\cite{B}, one also needs to
deal with the imaginary roots, but they will not feature in the present paper.
\end{remark}

We may define the ($q$-deformed) \emph{root vectors}:
\begin{equation*}
  E_{\pm \tbeta} \in \VVpm
\end{equation*}
for all $\tbeta \in \wDelta^{\mathrm{re},+}$ of~(\ref{eqn:real positive}) via
\begin{equation}\label{eqn:affine root generators}
  E_{\tbeta_k} =
  \begin{cases}
    T_{i_1} \dots T_{i_{k-1}} (e_{i_k}) &\text{if } k > 0 \\
    T_{i_0}^{-1} \dots T_{i_{k+1}}^{-1} (e_{i_k}) &\text{if } k \leq 0
  \end{cases}
\end{equation}
and
\begin{equation}\label{eqn:affine negative root generators}
  E_{-\tbeta_k} =
  \begin{cases}
    T_{i_1} \dots T_{i_{k-1}} (f_{i_k}) &\text{if } k > 0 \\
    T_{i_0}^{-1} \dots T_{i_{k+1}}^{-1} (f_{i_k}) &\text{if } k \leq 0
  \end{cases}
\end{equation}
where $\{T_i\}_{i \in \wI}$ determine Lusztig's affine braid group action~\cite{Lu} on $\VV$.

\begin{remark}\label{rem:comparison to Beck}
We note that $E_{-\tbeta}\in \VVm$ for $\tbeta \in \wDelta^{\mathrm{re},+}$ in~\cite{B} are defined via
\begin{equation}\label{eqn:negative via positive}
  E_{-\tbeta}:=\Omega(E_{\tbeta}) \,,
\end{equation}
where the $\BQ$-algebra anti-involution $\Omega$ of $\VV$ is determined by:
\begin{equation*}
  \Omega\colon e_i\mapsto f_i,\ f_i\mapsto e_i,\ \ph^{\pm 1}_i\mapsto \ph_i^{\mp 1},\ q\mapsto q^{-1}
  \qquad \forall\, i\in \wI \,.
\end{equation*}
Formulas~\eqref{eqn:affine negative root generators} and~\eqref{eqn:negative via positive} agree,
as $\Omega$ commutes with the affine braid group action:
\begin{equation}\label{eqn:omega vs T}
  \Omega\circ T_i = T_i\circ \Omega \qquad \forall\, i\in \wI \,.
\end{equation}
\end{remark}

According to~\cite[(5.28)]{NT} (based on~\cite[Proposition 7]{B}), we have
\begin{equation}\label{eqn:q comm affine}
  [E_{\pm \tbeta}, E_{\pm \talpha}]_q = E_{\pm \tbeta} E_{\pm \talpha} - q^{(\talpha, \tbeta)} E_{\pm \talpha} E_{\pm \tbeta}
  \in \BQ(q)^* \cdot E_{\pm (\talpha + \tbeta)}
\end{equation}
for any real positive affine roots $\talpha<\tbeta$ which both belong to either $\Delta^+ \times \BZ_{\geq 0}$
or $\Delta^- \times \BZ_{>0}$ and which also have the additional property that $\talpha + \tbeta$ is a positive
affine root whose decomposition as the sum of $\talpha$ and $\tbeta$ is \emph{minimal} in the sense that:
\begin{equation*}
  \not \exists \ \talpha',\tbeta' \in \wDelta^{\mathrm{re},+} \quad \text{s.t.} \quad
  \talpha < \talpha' < \tbeta' < \tbeta \quad \text{and} \quad \talpha + \tbeta = \talpha' + \tbeta' \,.
\end{equation*}
Let $U^\pm_q(+\infty)$ and $U^\pm_q(-\infty)$ denote the ``quarter'' subalgebras of $\VV$ generated by
$\{E_{\pm \tbeta_k} \,|\, k\geq 1\}$ and $\{E_{\pm \tbeta_k} \,|\, k\leq 0\}$, respectively. According
to~\cite[(5.35, 5.36)]{NT} (based on~\cite{B}), each of them admits a pair of opposite PBW decompositions:
\begin{equation}\label{eqn:quarter 01}
\begin{split}
  & U^\pm_q(+\infty) \ = \\
  & \qquad \
    \mathop{\bigoplus_{n_1,n_2,\dots \in \BN}}_{n_1+n_2+\dots < \infty}
    \BQ(q) \cdot E_{\pm \tbeta_1}^{n_1} E_{\pm \tbeta_2}^{n_2} \dots \ =
  \mathop{\bigoplus_{n_1,n_2,\dots \in \BN}}_{n_1+n_2+\dots < \infty}
    \BQ(q) \cdot \dots E_{\pm \tbeta_2}^{n_2} E_{\pm \tbeta_1}^{n_1} \,,
\end{split}
\end{equation}
\begin{equation}\label{eqn:quarter 02}
\begin{split}
  & U^\pm_q(-\infty) \ = \\
  & \mathop{\bigoplus_{n_{0},n_{-1},\dots \in \BN}}_{n_0+n_{-1}+\dots < \infty}
    \BQ(q) \cdot E_{\pm \tbeta_0}^{n_0} E_{\pm \tbeta_{-1}}^{n_{-1}} \dots \ =
  \mathop{\bigoplus_{n_{0},n_{-1},\dots \in \BN}}_{n_0+n_{-1}+\dots < \infty}
    \BQ(q) \cdot \dots E_{\pm \tbeta_{-1}}^{n_{-1}} E_{\pm \tbeta_{0}}^{n_{0}} \,.
\end{split}
\end{equation}


\subsection{Interplay of two algebras}
\

The relation between $\UU$ of Definition~\ref{def:quantum loop} and $\VV$ of Definition~\ref{def:quantum affine}
goes back to~\cite{B,B2,D3} and plays a crucial role in the theory of quantum affine algebras. In the present setup,
it amounts to the following result, cf.~\cite[Theorem 5.19]{NT}:

\begin{theorem}\label{thm:two presentations}
There exists an algebra isomorphism:
\begin{equation}\label{eqn:two presentations}
  \UU \ \iso \ \VV/(C-1)
\end{equation}
with $C$ of~\eqref{eqn:central c}, determined by the following assignment for all $i \in I$ and $d \in \BZ$:
\begin{equation}\label{eqn:Beck isomorphism}
\begin{split}
  & e_{i,d} \mapsto
  \begin{cases}
    o(i)^d E_{(\alpha_i,d)} &\text{if } d\geq 0 \\
    -o(i)^d E_{(\alpha_i,d)}\ph_i^{-1} &\text{if } d<0
  \end{cases}\,, \\
  &  f_{i,d} \mapsto
  \begin{cases}
    -o(i)^d \ph_i E_{(-\alpha_i,d)} &\text{if } d>0 \\
    o(i)^d E_{(-\alpha_i,d)} &\text{if } d\leq 0
  \end{cases} \,,
\end{split}
\end{equation}
where $o\colon I\to \{\pm 1\}$ is a map satisfying $o(i)o(j)=-1$ whenever $a_{ij}<0$.
\end{theorem}

The proof of this result is similar to that of~\cite[Theorem 5.19]{NT}, but it does essentially
utilize Proposition~\ref{prop:terminal-segment-2} as well as simplifies some arguments from~\cite{NT}.

\begin{proof}[Proof of Theorem~\ref{thm:two presentations}]
The isomorphism~\eqref{eqn:two presentations} was proved in~\cite[Theorem 4.7]{B2} with respect
to the following seemingly different formula:
\begin{equation}\label{eqn:Beck original formulas}
  e_{i,d} \mapsto o(i)^d T_{\widehat{\omega^\vee_i}}^{-d}(e_i) \,, \qquad
  f_{i,d} \mapsto o(i)^d T_{\widehat{\omega^\vee_i}}^d (f_i) \qquad \forall\, i\in I, d\in \BZ \,.
\end{equation}
Here, the aforementioned action of the affine braid group on $\VV$ has been extended to
the extended affine braid group by adding automorphisms $\{T_{\tau}\}_{\tau\in \CT}$:
  $$ T_\tau\colon e_i\mapsto e_{\tau(i)}, \ f_i\mapsto f_{\tau(i)}, \ \ph^{\pm 1}_i\mapsto \ph^{\pm 1}_{\tau(i)}
     \qquad \forall\, \tau\in \CT,i\in \wI \,, $$
which satisfy the relations $T_\tau T_i=T_{\tau(i)}T_\tau$ for any $\tau\in \CT$ and $i\in \wI$.

Therefore, it remains for us to show that~\eqref{eqn:Beck isomorphism} is equivalent
to~\eqref{eqn:Beck original formulas} by proving:
\begin{equation}\label{eqn:our vs Beck e}
  T_{\widehat{\omega^\vee_i}}^{-d}(e_i) =
  \begin{cases}
     E_{(\alpha_i,d)} &\text{if } d\geq 0 \\
     -E_{(\alpha_i,d)}\ph_i^{-1} &\text{if } d<0
  \end{cases} \,,
\end{equation}
\begin{equation}\label{eqn:our vs Beck f}
  T_{\widehat{\omega^\vee_i}}^d (f_i)=
  \begin{cases}
    -\ph_i E_{(-\alpha_i,d)} &\text{if } d>0 \\
    E_{(-\alpha_i,d)} &\text{if } d\leq 0
  \end{cases} \,.
\end{equation}
It suffices to prove only~\eqref{eqn:our vs Beck e} while~\eqref{eqn:our vs Beck f} then follows
as $\Omega$ commutes with the extended affine braid group action (due to~\eqref{eqn:omega vs T} and
$\Omega\circ T_\tau=T_\tau\circ \Omega$ for $\tau\in \CT$).

Fix $i\in I, d\geq 0$. According to~\eqref{eq:halves}, there is a unique $k\leq 0$ such that
\begin{equation}\label{eq:choice-k}
  (\alpha_i,d)=\tbeta_k=\beta_k=s_{i_0}s_{i_{-1}}\dots s_{i_{k+1}}(\alpha_{i_k}) \,.
\end{equation}
Invoking~\eqref{eq:omega-pi}, we claim that $\widehat{\omega_{i,d}}\in \weW$ has
a reduced decomposition of the form
\begin{equation}\label{eq:new-reduced}
  \widehat{\omega_{i,d}}=\tau s_{i_{k+1}}\dots s_{i_{-1}}s_{i_0} \qquad \mathrm{with} \quad \tau\in \CT \,.
\end{equation}
This follows from the equality of terminal sets $E_{s_{i_{k+1}}\dots s_{i_{-1}}s_{i_0}}=E_{\widehat{\omega_{i,d}}}$
(due to Proposition~\ref{prop:terminal-segment-2} and Theorem~\ref{thm:weyl to lyndon}) and the fact that $E_{x}=E_{y}$
iff $x^{-1}y\in \CT$ (already used in the proof of Theorem~\ref{thm:weyl to lyndon}).
Combining~\eqref{eq:choice-k} and~\eqref{eq:new-reduced}, we thus obtain
\begin{equation*}
  (\alpha_i,d) = s^{-1}_{i_0}s^{-1}_{i_{-1}}\dots s^{-1}_{i_{k+1}}(\alpha_{i_k}) =
  \widehat{\omega_{i,d}}^{-1}\tau(\alpha_{i_k})=\widehat{\omega_{i,d}}^{-1}(\alpha_{\tau(i_k)}) \,.
\end{equation*}
In view of~\eqref{eqn:coweight action}, this implies $\tau(i_k)=i$. Hence, we get:
\begin{equation*}
  E_{\tbeta_{k}} = T^{-1}_{i_0}T^{-1}_{i_{-1}} \dots T^{-1}_{i_{k+1}}(e_{i_k}) =
  T^{-1}_{\widehat{\omega_{i,d}}}\tau(e_{i_k})=  T^{-1}_{\widehat{\omega_{i,d}}}(e_i) \,.
\end{equation*}
According to Proposition~\ref{prop:length for lattice}, we have
$l(\widehat{\omega_{i,d}})=\sum_{j\in I} p_j l(\widehat{\omega^\vee_j})$, cf.~\eqref{eq:p-constants}, so that
\begin{equation*}
  T_{\widehat{\omega_{i,d}}} = \prod_{j\ne i} T^{p_j}_{\widehat{\omega_j^\vee}} \cdot T^{p_i}_{\widehat{\omega_i^\vee}} =
  \prod_{j\ne i} T^{p_j}_{\widehat{\omega_j^\vee}} \cdot T^{d}_{\widehat{\omega_i^\vee}} \,.
\end{equation*}
As $T^{\pm 1}_{\widehat{\omega_j^\vee}}(e_i)=e_i$ for $j\ne i$ by~\cite[Corollary 3.2]{B2}, we get the desired equality:
\begin{equation*}
  E_{(\alpha_i,d)}=E_{\tbeta_{k}} = T^{-1}_{\widehat{\omega_{i,d}}}(e_i) = T^{-d}_{\widehat{\omega_i^\vee}}(e_i) \,.
\end{equation*}

For $d<0$, the proof is similar and follows the same arguments as in~\cite{NT}.
\end{proof}


\subsection{PBW-type bases via quarter subalgebras}
\

The isomorphism \eqref{eqn:two presentations} does not intertwine the triangular decompositions
\eqref{eqn:triangular loop} and \eqref{eqn:triangular affine}. In fact, if we think of $\UU$ and
$\VV/(C-1)$ as one and the same algebra, then these two decompositions are ``orthogonal" as explained
in~\cite{NT}, cf.~\cite{EKP}. To this end, consider the following ``quarter'' subalgebras following~\cite[Lemmas 5--6]{B}:
\begin{align*}
  & U_q^+(L\fn^-) := \UUm \cap \VVg = \Big \{\text{subalgebra generated by } \sfe_{\tbeta_k}, k > 0 \Big\},\\
  & U_q^+(L\fn^+) := \UUp \cap \VVg = \Big \{\text{subalgebra generated by } \sfe_{\tbeta_k}, k \leq 0 \Big\},
\end{align*}
where we define $\sfe_{\tbeta_k}$ in accordance with~\eqref{eqn:Beck isomorphism} via:
\begin{equation}\label{eqn:twisted affine root generators}
  \sfe_{\tbeta_k} =
  \begin{cases}
    \ph_{-\hdeg \tbeta_k} E_{\tbeta_k} &\text{if } k > 0 \\
    E_{\tbeta_k} &\text{if } k \leq 0
\end{cases} \,.
\end{equation}
Henceforth, given a homogeneous element $z$ of degree $\left(\sum_{i\in I} k_i\alpha_i,d\right)\in Q\times \BZ$, set
\begin{equation*}
  \ph_{\pm \hdeg z} := \ph_{\pm \sum_{i\in I} k_i\alpha_i} =
  \prod_{i\in I} \ph^{\pm k_i}_i \, \in \UUo \,.
\end{equation*}
Formulas~\eqref{eqn:q comm affine} still hold when the $E_{\tbeta_k}$ are replaced with
the $\sfe_{\tbeta_k}$, since commuting $\ph$'s simply produces powers of $q$. Likewise, the
PBW decompositions~(\ref{eqn:quarter 01},~\ref{eqn:quarter 02}) imply that the subalgebras above
have the following PBW bases:
\begin{align}
  & U_q^+(L\fn^-) \ = \mathop{\bigoplus_{n_1,n_2,\dots \in \BN}}_{n_1+n_2+\dots < \infty}
    \BQ(q) \cdot \dots \sfe_{\tbeta_2}^{n_2} \sfe_{\tbeta_1}^{n_1} \,,
    \label{eqn:quarter 1} \\
  & U_q^+(L\fn^+) \ = \mathop{\bigoplus_{n_0,n_{-1},\dots \in \BN}}_{n_0+n_{-1}+\dots < \infty}
    \BQ(q) \cdot \dots \sfe_{\tbeta_{-1}}^{n_{-1}} \sfe_{\tbeta_0}^{n_0} \,.
    \label{eqn:quarter 2}
\end{align}
Likewise, we have PBW bases for analogous ``quarter'' subalgebras of $\VVl$:
\begin{align}
  & U_q^-(L\fn^-) := \UUm \cap \VVl \ =
    \mathop{\bigoplus_{n_0,n_{-1},\dots \in \BN}}_{n_0+n_{-1}+\dots < \infty}
    \BQ(q) \cdot \sfe_{-\tbeta_{0}}^{n_{0}} \sfe_{-\tbeta_{-1}}^{n_{-1}}\dots \,,
    \label{eqn:quarter 3} \\
  & U_q^-(L\fn^+) := \UUp \cap \VVl \ =
    \mathop{\bigoplus_{n_1,n_2,\dots \in \BN}}_{n_1+n_2+\dots < \infty}
    \BQ(q) \cdot \sfe_{-\tbeta_1}^{n_1} \sfe_{-\tbeta_2}^{n_2} \dots \,,
    \label{eqn:quarter 4}
\end{align}
where we define:
\begin{equation}\label{eqn:twisted negative affine root generators}
  \sfe_{-\tbeta_k} = \Omega(\sfe_{\tbeta_k}) =
  \begin{cases}
    E_{-\tbeta_k} \ph_{\hdeg \tbeta_k}  &\text{if } k > 0 \\
    E_{-\tbeta_k} &\text{if } k \leq 0
  \end{cases} \,.
\end{equation}

The following result allows to construct the PBW bases of $\UUpm$:

\begin{proposition}\label{prop:quarter}\cite[Proposition 5.23]{NT}
The multiplication map induces a vector space isomorphism:
\begin{equation*}
  U_q^+(L\fn^-) \otimes U_q^-(L\fn^-) \ \iso \ \UUm \,.
\end{equation*}
\end{proposition}

To make the presentation uniform, let us switch from $\tbeta_k$ of~\eqref{eqn:affine real roots} to $\beta_k$
of~\eqref{eqn:beta-roots}, so that $U_q^+(L\fn^-)$ and $U_q^-(L\fn^-)$ are generated by $\{\sfe_{-\beta_k}\}_{k\geq 1}$
and $\{\sfe_{-\beta_k}\}_{k\leq 0}$, respectively (note $\{-\beta_k\}_{k\in \BZ}=\Delta^-\times \BZ$).
Combining Proposition~\ref{prop:quarter} with the PBW decompositions~(\ref{eqn:quarter 1},~\ref{eqn:quarter 3}),
we obtain the PBW basis for $\UUm$, cf.~\cite[(5.69)]{NT}:

\begin{proposition}
(a) The subalgebra $\UUm$ admits the following PBW basis:
\begin{equation}\label{eqn:pbw basis loop}
  \UUm \ =
  \mathop{\bigoplus_{\cdots, n_{-1}, n_0, n_1,n_2, \dots \in \BN}}_{\dots + n_{-1} + n_0 + n_1+n_2+\dots < \infty}
    \BQ(q) \cdot \dots \sfe_{-\beta_{2}}^{n_{2}} \sfe_{-\beta_1}^{n_1} \sfe_{-\beta_0}^{n_0} \sfe_{-\beta_{-1}}^{n_{-1}} \dots
\end{equation}

\medskip
\noindent
(b) For any $s<r$, the root vectors $\sfe_{-\beta_s}$ and $\sfe_{-\beta_r}$ satisfy
\begin{equation}\label{eqn:orthogonal q comm general affine}
  \sfe_{-\beta_s} \sfe_{-\beta_r} - q^{(\beta_s, \beta_r)} \sfe_{-\beta_r} \sfe_{-\beta_s} \ \in
  \mathop{\bigoplus_{n_{r-1},\dots,n_{s+1}\in \BN}}
  \BQ(q) \cdot \sfe_{-\beta_{r-1}}^{n_{r-1}} \dots \sfe_{-\beta_{s+1}}^{n_{s+1}}
\end{equation}
where the sum is finite as it is taken over all tuples $n_{r-1},\dots,n_{s+1}\in \BN$ such that:
  $$ n_{r-1}\beta_{r-1}+\dots+n_{s+1}\beta_{s+1} = \beta_r+\beta_s \,. $$
\end{proposition}

\noindent
The analogous result also holds for $\UUp$ with $\sfe_{-\beta_s}$ replaced by $\sfe_{\beta_s}$.


\subsection{Proof of Theorem~\ref{thm:PBW quantum loop}}
\

Similarly to~\cite[Subsection 5.28]{NT}, we shall now see that Theorem \ref{thm:PBW quantum loop}
is equivalent to the PBW decomposition~(\ref{eqn:pbw basis loop}). Recall the reduced decomposition
of $\wmu$ produced by Theorem~\ref{thm:weyl to lyndon}, see Remark~\ref{rem:halves-into-total}, so
that the ordered set of roots
\begin{equation}\label{eqn:beta-order revisited}
  \dots < \beta_2 < \beta_1 < \beta_0 < \beta_{-1} < \cdots
\end{equation}
coincides with $\Delta^+ \times \BZ$ ordered in accordance with the bijection \eqref{eqn:associated word loop} via:
\begin{equation*}
  \dots < \ell(\obeta_2) < \ell(\obeta_1) < \ell(\obeta_0) < \ell(\obeta_{-1}) < \cdots
\end{equation*}
where for any $(\alpha,d)\in \Delta^+\times \BZ$ we set $\overline{(\alpha,d)} = (\alpha,-d)$.

Let $\varpi$ be the anti-involution of $\UU$ defined via
  $$ \varpi\colon e_{i,k}\mapsto f_{i,k} \,, \quad f_{i,k}\mapsto e_{i,k} \,, \quad \ph^\pm_{i,l}\mapsto \ph^{\pm}_{i,l} $$
for any $i\in I$, $k\in \BZ$, $l\in \BN$. Applying $\varpi$ to~(\ref{eqn:pbw basis loop}), we obtain:
\begin{equation}\label{eqn:pbw basis loop positive}
  \UUp \ =
  \bigoplus^{k\in \BN}_{\gamma_1\geq \dots \geq \gamma_k \in \Delta^+\times \BZ}
    \BQ(q) \cdot  \varpi(\sfe_{-\gamma_1}) \dots \varpi(\sfe_{-\gamma_k})
\end{equation}
with the above order on $\Delta^+ \times \BZ$ being~\eqref{eqn:beta-order revisited}.
On the other hand, due to~\eqref{eqn:orthogonal q comm general affine}, we obtain:
\begin{equation*}
  [\varpi(\sfe_{-\obeta'}),\varpi(\sfe_{-\obeta})]_q \ \in
  \mathop{\bigoplus^{k\in \BN}_{\ell(\beta)>\ell(\gamma_1)\geq \dots \geq \ell(\gamma_k) >\ell(\beta')}}_
    {\gamma_1 + \dots + \gamma_k = \beta+\beta'}
  \BQ(q) \cdot  \varpi(\sfe_{-\ogamma_1}) \dots \varpi(\sfe_{-\ogamma_k})
\end{equation*}
for any $\beta,\beta'\in \Delta^+\times \BZ$ such that $\obeta'<\obeta$, or equivalently $\ell(\beta')<\ell(\beta)$.
In particular, if $\beta+\beta'\in \Delta^+\times \BZ$ and $\beta,\beta'$ are minimal in the sense:
\begin{equation}\label{eqn:main minimal condition}
  \not \exists \ \alpha, \alpha' \in \Delta^+\times \BZ \quad \text{s.t.} \quad
  \obeta'<\oalpha'<\oalpha<\obeta \quad \text{and} \quad \alpha+\alpha'=\beta+\beta'
\end{equation}
we have
\begin{equation}\label{eqn:upsilon comm}
  [\varpi(\sfe_{-\obeta'}),\varpi(\sfe_{-\obeta})]_q\in \BQ(q)^*\cdot \varpi(\sfe_{-\obeta-\obeta'}) \,.
\end{equation}

We claim that Theorem~\ref{thm:PBW quantum loop} follows from~(\ref{eqn:pbw basis loop positive}).
To this end, it suffices to show:
\begin{equation}\label{eqn:loop coincidence}
  e_{\ell(\beta)} \in \BQ(q)^* \cdot \varpi(\sfe_{-\obeta})
\end{equation}
for any $\beta=(\alpha,d)\in \Delta^+\times \BZ$. We prove~(\ref{eqn:loop coincidence}) by induction on
the height of $\alpha\in \Delta^+$. The base case $\alpha=\alpha_i$ (with $i\in I$) is immediate, due
to~(\ref{eqn:Beck isomorphism},~\ref{eqn:twisted affine root generators},~\ref{eqn:twisted negative affine root generators}):
  $$ e_{\left[ i^{(d)} \right]} = e_{i,d} = \varpi(f_{i,d}) = \pm \varpi(\sfe_{(-\alpha_i,d)}) \,. $$
For the induction step, consider the costandard factorization $\ell = \ell_1\ell_2$ of $\ell = \ell(\alpha,d)$.
Since factors of standard loop words are standard, we have $\ell_1 = \ell(\gamma_1,d_1)$ and $\ell_2 = \ell(\gamma_2,d_2)$
for some $(\gamma_1,d_1), (\gamma_2,d_2) \in \Delta^+\times \BZ$ such that $\alpha = \gamma_1 + \gamma_2, d=d_1+d_2$.
By the induction hypothesis, we have $e_{\ell_k} \in \BQ(q)^* \cdot \varpi(\sfe_{(-\gamma_k,d_k)})$ for $k\in \{1,2\}$.
However, we note that $(\gamma_1,d_1) < (\alpha,d) < (\gamma_2,d_2)$ is a minimal decomposition in the sense
of~\eqref{eqn:main minimal condition}, according to Proposition \ref{prop:lyndon is minimal}. Therefore,
comparing~\eqref{eqn:quantum bracketing lyndon affine} with \eqref{eqn:upsilon comm}, we obtain:
\begin{equation*}
  e_\ell = [e_{\ell_1}, e_{\ell_2}]_q \in
  \BQ(q)^* \cdot \varpi([\sfe_{(-\gamma_2,d_2)}, \sfe_{(-\gamma_1,d_1)}]_q) =
  \BQ(q)^* \cdot \varpi(\sfe_{(-\alpha,d)})
\end{equation*}
as we needed to prove. This completes our proof of Theorem~\ref{thm:PBW quantum loop}.


\medskip

\section{Generalization to other orders}\label{sec:generalization}

In this Section, we generalize our main results to a larger family of orders on the alphabet
$\CI=\{i^{(d)}\}_{i\in I}^{d\in \BZ}$. Consider a collection of functions $\ff_i\colon \BZ\to \BR$ such that
\begin{itemize}

\item[$\bullet$]
  $\ff_i(0)=0$

\item[$\bullet$]
  all $\ff_i$ are strictly increasing unbounded functions

\item[$\bullet$]
  there are infinitely many $N$ (both in $\BR_{>0}$ and $\BR_{<0}$) such that there exist
  $\{N_i\}_{i\in I}\subset \BZ^I$ satisfying $\ff_i(N_i)=N$ for all $i\in I$.

\end{itemize}
We then define an order on $\CI$ (hence a lexicographic order on the loop words) via:
\begin{equation}\label{eqn:lex generalized}
  i^{(d)} < j^{(e)} \qquad \Longleftrightarrow \qquad
  \ff_i(d) > \ff_j(e)  \quad \text{  or  } \quad \ff_i(d) = \ff_j(e)  \text{ and } i<j \,.
\end{equation}
In the special case $\ff_i(d)=\frac{d}{c_i}$ (with $c_i\in \BZ_{>0}$) this
recovers~\eqref{eqn:lex affine intro} considered above.

\medskip
\noindent
$\bullet$
First, we need to update the filtration~\eqref{eqn:loop filtration} of the loop algebra $L\fn^+$. To this end,
we fix an increasing sequence $\{N^{(+,s)}\}_{s\in \BN}$ of non-negative numbers (respectively, a decreasing
sequence $\{N^{(-,s)}\}_{s\in \BN}$ of non-positive numbers) such that $N^{(\pm,0)}=0$ and there exist
$\{N^{(\pm,s)}_i\}_{i\in I}\subset (\pm \BN)^I$ satisfying $\ff_i(N^{(\pm,s)}_i)=N^{(\pm,s)}$ for all $i\in I$.
Then, we define $L^{(s)}\fn^+$ as the finite-dimensional Lie subalgebra of $L\fn^+$ generated~by
\begin{equation*}
  \left\{ e_i^{(d)} \,\Big|\, i\in I,\, N^{(-,s)}_i\leq d\leq N^{(+,s)}_i \right\} .
\end{equation*}
We also amend our former definition of $\CI^{(s)}$ in~\eqref{eqn:finite loop alphabet} by rather redefining
\begin{equation*}
  \CI^{(s)}=\left\{i^{(d)} \,\Big|\, i\in I,\, N^{(-,s)}_i\leq d\leq N^{(+,s)}_i \right\}
  \qquad \forall\, s\in \BN \,.
\end{equation*}
We may thus apply Definition~\ref{def:standard} to yield a notion of standard (Lyndon) loop words
with respect to $L^{(s)}\fn^+$, with the words made up only of $i^{(d)}\in \CI^{(s)}$.

\medskip
\noindent
$\bullet$
For $N$ as above, so that there exist $\{N_i\}_{i\in I}$ satisfying $N=\ff_i(N_i) \ \forall\, i$, we define
\begin{equation*}
  f_N\colon \Delta^+\to \BZ  \quad \mathrm{via} \quad  f_N(\alpha)=\sum_{i\in I} k_i\cdot N_i
  \quad \mathrm{for\ any} \quad \alpha=\sum_{i\in I} k_i\alpha_i\in \Delta^+ \,.
\end{equation*}
With this definition at hand, the subalgebra $L^{(s)}\fn^+$ can be explicitly written as
\begin{equation*}
  L^{(s)}\fn^+=\bigoplus_{\alpha \in \Delta^+} \bigoplus_{d=f_{N^{(-,s)}}(\alpha)}^{f_{N^{(+,s)}}(\alpha)} \BQ \cdot e_\alpha^{(d)} \,.
\end{equation*}
Then, the results of Proposition~\ref{prop:classification} and Proposition~\ref{prop:l1} still hold true with
the only change that $-sf(\alpha)\leq d\leq sf(\alpha)$ is replaced with $f_{N^{(-,s)}}(\alpha)\leq d\leq f_{N^{(+,s)}}(\alpha)$.

\medskip
\noindent
$\bullet$
As before, we call a loop word $ w=\left[i_1^{(d_{1})} \dots \, i_n^{(d_{n})}\right]$ \emph{exponent-tight}
if~\eqref{eq:assumption} holds. Then, the results of Theorem~\ref{thm:ExpRule}, Lemma~\ref{lem:Uniquexponenttight},
Proposition~\ref{prop:firstletter}, and Proposition~\ref{prop:coherent} still hold true (the proofs are the same).
Therefore, we still have a bijection~\eqref{eqn:associated word loop}
\begin{equation*}
  \ell \colon \Delta^+\times \BZ \ \iso \ \Big\{\text{standard Lyndon loop words}\Big\}
\end{equation*}
satisfying property~\eqref{eqn:property lyndon} with $s = \infty$ as well as
Theorem~\ref{thm:ExpRule} and Proposition~\ref{prop:firstletter}.

\medskip
\noindent
$\bullet$
On the other hand, the periodicity of Proposition~\ref{prop:peridocitiy} no longer holds in this generality.
Instead, we can only express $\ell(\alpha,f_N(\alpha))$ via $\ell(\alpha,0)$:

\begin{lemma}
If $\ell(\alpha,0)=[i_1^{(0)} \dots\, i_n^{(0)}]$, then $\ell(\alpha,f_N(\alpha))=[i_1^{(N_{i_1})} \dots \, i_n^{(N_{i_n})}]$.
\end{lemma}

\begin{proof}
First, we note that Theorem~\ref{thm:ExpRule} together with Proposition~\ref{prop:firstletter} for $N>0$
(respectively, Remark~\ref{rem:first-letter-minus} for $N<0$) guarantee that the multiset of letters constituting
$\ell(\alpha,f_N(\alpha))$ is exactly $\{i_1^{(N_{i_1})}, \dots, i_n^{(N_{i_n})}\}$. Indeed, assuming the contradiction
for some $N>0$ (the case $N<0$ is treated analogously), there exists $0\leq d<f_N(\alpha)$ such that $\ell(\alpha,d+1)$
starts with $i_k^{(N_{i_k}+1)}$ for some $k$. As the sum of exponents equals $d+1\leq f_N(\alpha)$, the word
$\ell(\alpha,d+1)$ and hence $\ell(\alpha,d)$ also contains a letter $i_l^{(e)}$ with $e<N_{i_l}$. This provides
a contradiction with Proposition~\ref{prop:firstletter}, since $i_l^{(e+1)} > i_k^{(N_{i_k}+1)}$.

On the other hand, we note that $i^{(N_i)}<j^{(N_j)}$ iff $i<j$, which guarantees that the loop word
$[j_1^{(N_{j_1})} \dots \, j_n^{(N_{j_n})}]$ is (standard) Lyndon iff the loop word $[j_1^{(0)} \dots \, j_n^{(0)}]$ is
(standard) Lyndon, cf.\ the proof of Proposition~\ref{prop:peridocitiy}. This completes the proof.
\end{proof}

We can now prove the following slight generalization of Corollary~\ref{cor:first letter}:

\begin{lemma}\label{lem:first letter generalized}
Fix $N$, $\{N_i\}_{i\in I}$, $\alpha\in \Delta^+$, $f_N(\alpha)\in \BZ$ as above.

\medskip
\noindent
(a) For $d>f_N(\alpha)$, the loop word $\ell(\alpha,d)$ starts with some $j^{(e)}$ such that $e>N_j$.

\medskip
\noindent
(b) For $d\leq f_N(\alpha)$, the loop word $\ell(\alpha,d)$ starts with some $j^{(e)}$ such that $e\leq N_j$.
\end{lemma}

\begin{proof}
Let $\ell(\alpha,d)=[i_1^{(d_1)} \dots \, i_n^{(d_n)}]$. Then $i_1^{(d_1)}\leq i_r^{(d_r)}$
and so $\ff_{i_1}(d_1)\geq \ff_{i_r}(d_r)$ for any $r$. If $d_1\leq N_{i_1}$, then we would have
  $\ff_{i_r}(N_{i_r}) = \ff_{i_1}(N_{i_1}) \geq \ff_{i_1}(d_1) \geq \ff_{i_r}(d_r)$
and so $d_r\leq N_{i_r}$ for any $r$. That would imply $d=\sum_{r=1}^n d_r\leq f_N(\alpha)$, a contradiction.

To prove~(b), we note that $i_1^{(d_1)}\geq i_r^{(d_r+1)}$ for any $r$ by Theorem~\ref{thm:ExpRule}.
If $d_1>N_{i_1}$, then we would have
  $\ff_{i_r}(N_{i_r}) = \ff_{i_1}(N_{i_1}) < \ff_{i_1}(d_1) \leq \ff_{i_r}(d_r+1)$
so that $d_r\geq N_{i_r}$ for all $r$. That would imply $d=\sum_{r=1}^n d_r>f_N(\alpha)$, a contradiction.
\end{proof}

\noindent
$\bullet$
The major difference will take place in the generalization of Theorem~\ref{thm:weyl to lyndon} to the present setup.
As the periodicity of Proposition~\ref{prop:peridocitiy} no longer holds, the bi-infinite sequence $\{i_k\}_{k\in \BZ}$
of~\eqref{eqn:tau-twisted sequence} shall rather be constructed as a limit of finite sequences.
Explicitly, to define $\{i_{k}\}_{k\leq 0}$, instead of $L$ from~\eqref{eqn:L-chunk} we shall consider
  $$ L^{[s]}=\Big\{(\alpha,d) \,\Big|\, \alpha\in \Delta^+, 0\leq d<f_{N^{(+,s)}}(\alpha)\Big\} \qquad \forall\, s\in \BZ_{>0}. $$
These ``blocks'' can be identified with the following terminal sets:
  $$ L^{[s]}=E_{\widehat{\mu^{(+,s)}}} \quad \mathrm{with} \quad
     \mu^{(+,s)}=\sum_{i\in I} N_i^{(+,s)}\omega^\vee_i .$$
Arguing as in the proof of Theorem~\ref{thm:weyl to lyndon}, we get a reduced
decomposition~(\ref{eqn:reduced-rho}) of $\widehat{\mu^{(+,s)}}$ such that the ordered finite sequence
$\beta_0<\beta_{-1}<\dots<\beta_{1-l(\widehat{\mu^{(+,s)})}}$ coincides with $L^{[s]}$ ordered via~(\ref{eqn:order-L}).
By uniqueness, such a sequence for $L^{[s+1]}$ refines the one for $L^{[s]}$. Furthermore, the roots
$\beta_k=s_{i_0} s_{i_{-1}} \dots s_{i_{k+1}}(\alpha_{i_k})$ for $k\leq 0$ are all distinct and satisfy
$\{\beta_k\}_{k\leq 0}=\Delta^+\times \BZ_{\geq 0}$. The construction of $\{i_{k}\}_{k>0}$ is similar.

\medskip
\noindent
$\bullet$
With the above update of Theorem~\ref{thm:weyl to lyndon}, the results of Propositions~\ref{prop:terminal-segment},~\ref{prop:terminal-segment-2}
still hold (for any fixed $i\in I, d\in \BZ$), once $\{p_j\}_{j\in I}$ from~\eqref{eq:p-constants} are rather redefined via:
\begin{equation*}
  p_j \ \ \mathrm{is\ the\ unique\ integer\ satisfying} \ \ j^{(-p_j)}\geq i^{(-d)}>j^{(-p_j+1)} \,
\end{equation*}
with $i\in I,d\geq 0$ fixed. This recovers~\eqref{eq:p-constants} if $\ff_j(d)=\frac{d}{c_j}$ for all $j$ (with $c_j\in \BZ_{>0}$).

\medskip
\noindent
$\bullet$
With the above update, the analogue of~\eqref{eq:new-reduced} and the paragraph afterwards hold, implying~(\ref{eqn:our vs Beck e},~\ref{eqn:our vs Beck f}).
Thus, the main result of Section~\ref{sec:quantum}, the construction of PBW bases of $U_q(L\fn^+)$ from Theorem~\ref{thm:PBW quantum loop} still holds
(with $e_\ell,e_w$ as in Definition~\ref{def:loop quantum bracketing}).


\medskip

\appendix
\section{Computer code}\label{sec:appendix}

In this Appendix, we present some interesting examples of standard Lyndon loop words that
nicely illustrate the key properties of Theorem~\ref{thm:ExpRule} and Proposition~\ref{prop:firstletter}.
We also provide a link to our code used to evaluate standard Lyndon loop words.


\subsection{Examples}
\

The first version of our code did not use the key results (Theorem~\ref{thm:ExpRule} and Proposition~\ref{prop:firstletter}),
but was rather based on Remark~\ref{rem:silly-generalization-NT}, which is a simple generalization of~\cite[Proposition 2.26]{NT}.
Thus, when evaluating $\ell(\alpha,d)$, the code simply goes through all the ways to split $\alpha$ into an ordered sum of
simple roots, and distribute $d$ between the exponents of these simple roots while satisfying~\eqref{eq:silly-assumption}.
In the table below, we present examples of standard Lyndon loop words computed through this code
(which also nicely illustrate the results of Theorem~\ref{thm:ExpRule} and Proposition~\ref{prop:firstletter}).

\begin{center}
\begin{tabular}{ |c|c|c|c|c|c|c|c|c| }
\hline
  Type & Order & Weights & $d$ & $\ell(\theta,d)$ & $\ell(\theta,d+1)$ & $\ell(\theta,d+2)$ \\
\hline
 $A_4$ & 1234 & 1 1 1 1 & 0 &
 $[1^{(0)} 2^{(0)} 3^{(0)} 4^{(0)}]$ &
 $[4^{(1)} 3^{(0)} 2^{(0)} 1^{(0)}]$ &
 $[3^{(1)} 2^{(0)} 1^{(0)} 4^{(1)}]$\\

 $A_5$ & 51324 & 4 3 1 8 5 & 19 &
 $[3^{(1)} 4^{(8)} 5^{(4)} 2^{(3)} 1^{(3)}]$ &
 $[1^{(4)} 2^{(3)} 3^{(1)} 4^{(8)} 5^{(4)}]$ &
 $[5^{(5)} 4^{(8)} 3^{(1)} 2^{(3)} 1^{(4)}]$ \\

 $B_2$ & 21 & 7 8 & 18 &
 $[1^{(6)} 2^{(6)} 2^{(6)}]$&
 $[2^{(7)} 1^{(6)} 2^{(6)}]$&
 $[2^{(7)} 2^{(7)} 1^{(6)}]$\\

 $B_3$ & 123  & 4 3 1 & 10 &
 $[2^{(3)} 1^{(3)} 3^{(1)} 3^{(1)} 2^{(2)}]$&
 $[2^{(3)} 3^{(1)} 3^{(1)} 2^{(3)} 1^{(3)}]$&
 $[1^{(4)} 2^{(3)} 3^{(1)} 3^{(1)} 2^{(3)}]$\\

 $C_3$ & 312  & 4 3 6 & 8 &
 $[1^{(2)} 2^{(1)} 1^{(2)} 2^{(1)} 3^{(2)}]$&
 $[3^{(3)} 2^{(1)} 2^{(1)} 1^{(2)} 1^{(2)}]$&
 $[2^{(2)} 1^{(2)} 3^{(3)} 2^{(1)} 1^{(2)}]$\\

 $C_3$ & 321 &  1 10 3 & 17 &
 $[2^{(8)} 1^{(0)} 3^{(2)} 2^{(7)} 1^{(0)}]$&
 $[2^{(8)} 1^{(0)} 2^{(8)} 1^{(0)} 3^{(2)}]$&
 $[2^{(9)} 1^{(0)} 3^{(2)} 2^{(8)} 1^{(0)}]$\\

 $D_4$ & 3124 & 4 3 7 5 & 8 &
 $[3^{(3)} 2^{(1)} 1^{(1)} 4^{(2)} 2^{(1)}]$&
 $[1^{(2)} 2^{(1)} 4^{(2)} 3^{(3)} 2^{(1)}]$&
 $[3^{(4)} 2^{(1)} 4^{(2)} 1^{(2)} 2^{(1)}]$\\

 $G_2$ & 21 & 2 3 & 11 &
 $[2^{(3)} 1^{(2)} 2^{(2)} 2^{(2)} 1^{(2)}]$&
 $[2^{(3)} 1^{(2)} 2^{(3)} 1^{(2)} 2^{(2)}]$&
 $[2^{(3)} 2^{(3)} 1^{(2)} 2^{(3)} 1^{(2)}]$\\
\hline
\end{tabular}
\end{center}

\medskip
\noindent
Let us clarify the conventions in this table:
\begin{itemize}

\item[$\bullet$]
  In the column ``Order'', the elements $i\in I$ are listed in the increasing order;

\item[$\bullet$]
  In the column ``Weights'', the weights $c_i$ are listed with $i$ ordered as in~\cite{Sa};

\item[$\bullet$]
  In all these examples, we choose to consider only the highest root $\alpha=\theta$.

\end{itemize}

\medskip
\noindent
Let us also provide examples of standard Lyndon loop words for the remaining exceptional types
(these were evaluated using our second code presented~below):

\begin{center}
\begin{tabular}{|c|c|c|c|c|}
\hline
  Type & Order & Weights & $d$ & $\ell(\theta,d)$ \\
\hline
 $F_4$ & 1234 & 1 2 3 2 & 17 & $[3^{(3)} 2^{(1)} 2^{(1)} 4^{(2)} 3^{(3)} 2^{(1)} 1^{(0)} 3^{(3)} 2^{(1)} 1^{(0)} 4^{(2)}]$ \\

 $F_4$ & 1234 &  1 2 3 2 & 18 & $[2^{(2)} 1^{(0)} 3^{(3)} 2^{(1)} 1^{(0)} 4^{(2)} 3^{(3)} 2^{(1)} 2^{(1)} 3^{(3)} 4^{(2)}]$ \\

 $E_6$ & 142653 &  1 2 1 2 2 1 & 9 & $[5^{(2)} 4^{(1)} 3^{(1)} 6^{(0)} 2^{(1)} 1^{(0)} 3^{(1)} 2^{(1)} 4^{(1)} 3^{(1)} 6^{(0)}]$ \\

 $E_6$ & 142653 &  1 2 1 2 2 1 & 10 & $[6^{(1)} 3^{(1)} 2^{(1)} 1^{(0)} 4^{(1)} 3^{(1)} 2^{(1)} 5^{(2)} 4^{(1)} 3^{(1)} 6^{(0)}]$ \\

 $E_7$ & 1234567 &  4 5 3 7 3 2 5 & 25
 & $[4^{(3)} 5^{(1)} 6^{(0)} 3^{(1)} 7^{(2)} 4^{(2)} 2^{(2)} 1^{(1)} 3^{(1)} 2^{(2)} 4^{(3)} 5^{(1)} 6^{(0)} 3^{(1)} 7^{(2)} 4^{(2)} 5^{(1)}]$ \\

 $E_7$ & 1234567 &  4 5 3 7 3 2 5 & 26
 & $[4^{(3)} 5^{(1)} 3^{(1)} 7^{(2)} 4^{(2)} 2^{(2)} 1^{(1)} 3^{(1)} 2^{(2)} 4^{(3)} 5^{(1)} 6^{(0)} 3^{(1)} 7^{(2)} 4^{(3)} 5^{(1)} 6^{(0)}]$ \\

 $E_8$ & 14572386 &  1 32 13 3 10 9 6 15 & 46 & $[8^{(3)} 5^{(1)} 4^{(0)} 6^{(1)} 5^{(1)} 3^{(2)} 4^{(0)} 7^{(1)} 6^{(1)} 5^{(1)} 2^{(6)} 1^{(0)} 3^{(2)} 4^{(0)} 2^{(6)} 3^{(2)} $\\
 & & & & $ 8^{(3)} 5^{(1)} 4^{(0)} 6^{(1)} 5^{(1)} 3^{(2)} 4^{(0)} 7^{(1)} 6^{(1)} 5^{(1)} 8^{(2)} 2^{(6)} 1^{(0)}]$ \\
  $E_8$ & 14572386 &  1 32 13 3 10 9 6 15 & 47 & $[8^{(3)} 5^{(1)} 4^{(0)} 6^{(1)} 5^{(1)} 3^{(2)} 7^{(1)} 6^{(1)} 2^{(6)} 1^{(0)} 8^{(3)} 5^{(1)} 4^{(0)} 6^{(1)} 5^{(1)} 3^{(2)}$ \\
 & & & &$4^{(0)} 7^{(1)} 2^{(6)} 3^{(2)} 8^{(3)} 5^{(1)} 4^{(0)} 6^{(1)} 5^{(1)} 3^{(2)} 4^{(0)} 2^{(6)} 1^{(0)}]$ \\
\hline
\end{tabular}
\end{center}
\medskip


\subsection{The code}
\

The second version of our code was written using Proposition~\ref{prop:classification} as well as
Proposition~\ref{prop:firstletter} which provides an inductive way to compute exponents of $\ell(\alpha,d)$.
This drastically improves the code performance, allowing us to compute words for much larger values of
the degree $d$ and the weights $c_i$. This code can be used at the following clickable
link (the interested reader can use this code to check the results of this note as well as to compute
standard Lyndon loop words):\footnote{The user should press the ``Run'' button and they will see the instructions
and a small example afterwards. Type in the input in the console afterwards, following the instructions. Names of
Lie algebra types for input are: A, B, C, D, G2, F4, E6, E7, E8. This code was written using C++23.}
\begin{itemize}[leftmargin=0.7cm]

\item
 \href{https://onlinegdb.com/efjH329LN}{C++ Code 2}

\end{itemize}


\subsection{Divisible weights in type A}
\

In this Subsection, we consider a special setup for type $A_n$ (naturally generalizing \cite[Section 7.3]{NT}):
the order is $1<2<\cdots<n$, and the weights $c_1,c_2,\ldots,c_{n}\in \BZ_{>0}$ are such that $c_{i}$ divides $c_{i+1}$
for any $1\leq i<n$. By induction on $n$ and the periodicity of Proposition~\ref{prop:peridocitiy}, it suffices to
evaluate $\ell(\theta,d)$ for $0<d\leq c_1+\dots+c_{n}$. Let $a^{(k)}$ be the first letter of the standard Lyndon
loop word $\ell(\theta,d)$. Then, we have:
\begin{multline*}
  \ell(\theta,d)=\left[ a^{(k)} (a-1)^{(k_2)} \dots \, 1^{(k_a)}\, (a+1)^{(k_{a+1})}\, (a+2)^{(k_{a+2})} \dots\, n^{(k_n)} \right] \\
  \mathrm{with} \quad k_i=\left\lceil k\cdot \frac{c_{a-i+1}}{c_a}-1\right\rceil  \ \ \mathrm{if} \ \ 1<i\leq a\,,
  \quad k_i=k\cdot \frac{c_i}{c_a} \ \ \mathrm{if} \ \ a<i\leq n \,.
\end{multline*}

It thus suffices to describe the first letter $a^{(k)}$. This is uniquely determined by a sequence encoding
the underlying element $a\in \{1,\ldots,n\}$ as $d$ increases from $1$ up to $c_1+c_2+\dots+c_n$. Indeed,
the exponent $k$ of $a$ (as well as the exponent of any other~$i$) is then equal to the number of times this
$a$ (respectively $i$) appears among the first $d$ terms of that sequence, due to Proposition~\ref{prop:firstletter}.
One can depict this sequence by a table placing each $n$ in the top of a new column to the right and then
moving top-to-bottom until getting to the next $n$. Let us now present a general rule for the construction of this table:
\begin{enumerate}

\item[1.]
At the first step, place $n$ in the top-left corner;

\item[2.]
At the $i$-th step (with $2\leq i\leq n$), copy the current table and paste it to the right
$\frac{c_{n-i+2}}{c_{n-i+1}}-1$ times. After that, add an extra entry $n-i+1$ at the bottom of the right-most column;

\item [3.] Copy the resulting table and paste it to the right $c_1-1$ times.

\end{enumerate}
Let us illustrate it with some examples. For $n=4$ and $c_1=1, c_2=2, c_3=6, c_4=12$, the sequence is
4 4 3 4 4 3 4 4 3 2 4 4 3 4 4 3 4 4 3 2 1, and so the table is:
\begin{align*}
\setlength{\tabcolsep}{2pt}
\begin{tabular}{c c c c c c c c c c c c }
    4&4& 4& 4&4& 4& 4& 4& 4& 4 & 4 &4\\
     &3 & &3 & &3& & 3 &&3&&3\\
    & & & &&2&&& &&& 2\\
    & & & & &&&&&&& 1
\end{tabular}
\end{align*}
For $n=3$ and $c_1=1, c_2=3, c_3=15$, the sequence is 3 3 3 3 3 2 3 3 3 3 3 2 3 3 3 3 3 2 1,
which is thus encoded by the following table:
\begin{align*}
\setlength{\tabcolsep}{2pt}
\begin{tabular}{c c c c c c c c c c c c c c c}
   3& 3& 3& 3& 3& 3&3& 3& 3&3& 3& 3&3& 3& 3\\
    & & & & 2& && & &2& & && & 2\\
    & & & & & && & && & && & 1\\
\end{tabular}
\end{align*}
Likewise, for $n=4$ and $c_1=1, c_2=3, c_3=9, c_4=27$, we get the following table:
\begin{align*}
\setlength{\tabcolsep}{2pt}
\begin{tabular}{c c c c c c c c c c c c c c c c c c c c c c c c c c c c}
    4& 4& 4& 4& 4& 4& 4& 4& 4& 4& 4& 4& 4& 4& 4& 4& 4& 4& 4& 4& 4& 4& 4& 4& 4& 4& 4\\
    & & 3& & & 3& & & 3& & &3& & & 3& & & 3& & &3& & & 3& & & 3\\
    & & & & & & & & 2& & && & & & & & 2& & & & & && & & 2\\
    & & & & & & & & & & && & & & & & & & & & & & && & 1\\
\end{tabular}
\end{align*}

\begin{remark}
We note that similar tables can also be produced for other classical types $B_n, C_n, D_n$ with the order $1<2<\dots<n$.
By induction on $n$, the periodicity of Proposition~\ref{prop:peridocitiy}, and the $A$-type case treated above, it suffices
to evaluate $\ell(\alpha,d)$ for the roots $\alpha=m_1\alpha_1+\dots+m_n\alpha_n\in \Delta^+$ with $m_1,\ldots,m_{n}\geq 1$
and $0<d\leq m_1c_1+\dots+m_nc_{n}$. The only difference between the corresponding tables and those for $A_n$-type, is that now
when adding each $i$ we shall be adding it $m_i$ times. Explicitly, the corresponding table is constructed by the following algorithm:
\begin{enumerate}

\item[1.]
At the first step, build a column of height $m_n$ with all entries equal to $n$;

\item[2.]
At the $i$-th step (with $2\leq i\leq n$), copy the current table and paste it to the right
$\frac{c_{n-i+2}}{c_{n-i+1}}-1$ times. After that, add $m_{n-i+1}$ times the number $n-i+1$ at the bottom of the right-most column;

\item [3.] Copy the resulting table and paste it to the right $c_1-1$ times.

\end{enumerate}
As an example, consider type $C_4$ with the weights $c_1=1, c_2=2, c_3=6, c_4=12$, and
$\alpha=2\alpha_1+2\alpha_2+2\alpha_3+\alpha_4=\theta$. Then, we get the following table:
\begin{align*}
\setlength{\tabcolsep}{2pt}
\begin{tabular}{c c c c c c c c c c c c }
    4&4& 4& 4&4& 4& 4& 4& 4& 4 & 4 &4\\
     &3 & &3 & &3& & 3 &&3&&3\\
     &3 & &3 & &3& & 3 &&3&&3\\
    & & & &&2&&& &&& 2\\
    & & & &&2&&& &&& 2\\
    & & & & &&&&&&& 1\\
    & & & & &&&&&&& 1
\end{tabular}
\end{align*}
The multiset of all letters appearing in $\ell(\alpha,d)$ is easily determined by this table:
if $p_i=m_i d_i + r_i\ (d_i\in \BN, 0\leq r_i<m_i)$ denotes the number of times $i$ appears among
the first $d$ terms of the table, then $r_i$ exponents of $i$ are $d_i+1$ and the rest are $d_i$.
\end{remark}

\end{document}